\documentclass[twoside,10pt]{article}

\usepackage[]{hyperref}                        % book-mark
\usepackage{amsmath}
\usepackage{amsthm}
\usepackage{bm,latexsym}
\usepackage{amsfonts}
\usepackage{indentfirst}
\usepackage{amsxtra}
\usepackage{amssymb}
\newtheorem{Lemma}{Lemma}[section]
\newtheorem{Theorem}{Theorem}

\theoremstyle{remark}
\newtheorem{Remark}{Remark}[section]

\numberwithin{equation}{section}

\begin{document}

%---------title-----------------------------------------------------------------------------------------
\title{\bf
Inflow Problem for the One-dimensional Compressible Navier-Stokes Equations under Large Initial Perturbation}
\author{{\bf  Lili Fan}\\
Department of Mathematic and Physics,
Wuhan Polytechnic University\\
 Wuhan 430023, China\\[2mm]
{\bf Hongxia Liu}\\
Department of Mathematics, Jinan University\\
Guangzhou 510632, China\\[2mm]
{\bf Tao Wang}\thanks{Corresponding author. Email address: pdewangtao@gmail.com}\\
School of Mathematics and Statistics, Wuhan University\\
Wuhan 430072, China\\[2mm]
{\bf Huijiang Zhao}\\
School of Mathematics and Statistics, Wuhan University\\
Wuhan 430072, China}

\date{} \maketitle

%--------abstract----------------------------------------------------------------------------------------

\begin{abstract}
This paper is concerned with the inflow problem for the one-dimensional compressible Navier-Stokes equations. For such a problem, Matsumura and Nishihara showed in [A. Matsumura and K. Nishihara, Large-time behaviors of solutions to an inflow problem in the half space for a one-dimensional system of compressible viscous gas. Comm. Math. Phys. 222 (2001), 449-474] that there exists boundary layer solution to the inflow problem and both the boundary layer solution, the rarefaction wave, and the superposition of boundary layer solution and rarefaction wave are nonlinear stable under small initial perturbation. The main purpose of this paper is to show that similar stability results for the boundary layer solution and the supersonic rarefaction wave still hold for a class of large initial perturbation which can allow the initial density to have large oscillation. The proofs are given by an elementary energy method and the key point is to deduce the desired lower and upper bounds on the density function.

\bigbreak
 \noindent {\bf Keywords}: Compressible Navier-Stokes equations; Boundary layer solution; Inflow problem;
 Rarefaction wave; Large initial perturbation; Large density oscillation; Continuation process.
\end{abstract}

\tableofcontents

\section{Introduction}

This paper is concerned with the large time behaviors of solutions to the inflow problem for one-dimensional compressible Navier-Stokes equations on the half line $\mathbb{R}_+=(0,+\infty)$, which is an initial-boundary value problem in Eulerian coordinates:
\begin{equation} \label{EulerI}
  \begin{cases}
    \rho_t+(\rho u)_{{x}}=0,&\textrm{in}\ \mathbb{R}_+ \times\mathbb{R}_+,\\[2mm]
    (\rho u)_t+(\rho u^2 +\tilde{p})_{{x}}=\mu u_{{x}{x}},&\textrm{in}\ \mathbb{R}_+ \times\mathbb{R}_+,\\[2mm]
    (\rho,u)|_{{x}=0}=(\rho_-,u_-),&u_->0,\\[2mm]
    (\rho,u)(0,{x})=(\rho_0,u_0)({x})\to (\rho_+,u_+), \quad &\textrm{as}\ {x}\to+\infty.
  \end{cases}
\end{equation}
Here, $\rho(>0)$, $u$, and  $\tilde{p}=\tilde{p}(\rho) =\rho^{\gamma}$ with $\gamma\geq 1$ being the adiabatic exponent
are, respectively, the density, the velocity, and the pressure, while the viscosity coefficient $\mu(>0)$,
$\rho_{\pm}(> 0)$ and $u_{\pm}$ are constants.

We assume that the initial data $(\rho_0(x), u_0(x))$ satisfy the boundary condition \eqref{EulerI}$_3$ as a compatibility condition, i.e.
$$
\rho_0(0)=\rho_-,\quad u_0(0)=u_-.
$$
The assumption $u_- > 0$ implies that, through the boundary $x=0$ the fluid
with the density $\rho_-$ flows into the region $\mathbb{R}_+$,
and hence the problem (\ref{EulerI}) is called the inflow problem.
The cases of $u_-= 0$ and $u_- < 0$, the problems where the condition
$\rho(t,0) =\rho_- $ is removed, are called the impermeable wall problem and the outflow problem, respectively.

For the case of $u_- > 0$,  as in~\cite{M-N2001}, the inflow problem (\ref{EulerI}) can then be transformed to the problem in the Lagrangian coordinates:
\begin{equation} \label{LagrangeI}
  \begin{cases}
    v_t-u_x=0,   &x>s_- t,\ t>0, \\[2mm]
    u_t+p(v)_x=\mu\left(\frac{u_x}{v}\right)_x, & x>s_- t,\ t>0, \\[2mm]
    (v,u)|_{x=s_-t}=(v_-,u_-),&u_->0,\\[2mm]
    (v,u)|_{t=0}=(v_0,u_0)(x)\to (v_+,u_+),&\textrm{as}\ x\to +\infty,
  \end{cases}
\end{equation}
where
\begin{equation}
  p(v)=v^{-\gamma},\quad v=\frac{1}{\rho},\quad v_{\pm}=\frac{1}{\rho_{\pm}},\quad s_-=-\frac{u_-}{v_-}<0.
\end{equation}
The characteristic speeds of the corresponding hyperbolic system of~(\ref{LagrangeI}) are
\begin{equation}
  \lambda_1=-\sqrt{-p'(v)},\quad \lambda_2=\sqrt{-p'(v)},
\end{equation}
and the sound speed $c(v)$ is defined by
\begin{equation}
  c(v)=v\sqrt{-p'(v)}=\sqrt{\gamma}v^{-\frac{\gamma-1}{2}}.
\end{equation}
Comparing $|u|$ with $c(v)$, we divide the phase space $\mathbb{R}_+\times\mathbb{R}_+$ into three regions:

\begin{equation*}
  \begin{aligned}
    \Omega_{sub}&:=\left\{(v,u);\ |u|<c(v),v>0,u>0\right\},\\[2mm]
    \Gamma_{trans}&:=\left\{(v,u);\ |u|=c(v),v>0,u>0\right\},\\[2mm]
    \Omega_{super}&:=\left\{(v,u);\ |u|>c(v),v>0,u>0\right\},
  \end{aligned}
\end{equation*}
which are called the subsonic, transonic and supersonic regions, respectively.

In the phase plane, we denote the curves through a point $(v_1,u_1)$
\begin{equation*}
    BL(v_1,u_1):=\left\{(v,u)\in\mathbb{R}_+\times\mathbb{R}_+;\ \frac{u}{v}=\frac{u_1}{v_1}\right\}
\end{equation*}
and
\begin{equation*}
    R_i(v_1,u_1):=\left\{(v,u)\in\mathbb{R}_+\times\mathbb{R}_+;\ u=u_1-\int_{v_1}^v \lambda_i(s)ds,\ u>u_1\right\}\quad (i=1,2)
\end{equation*}
as the boundary layer line and the $i$-rarefaction wave curve, respectively.

For the precise description of the large time behaviors of solutions to the initial-boundary value problem in the half line for the one-dimensional isentropic model system \eqref{EulerI}$_1$-\eqref{EulerI}$_2$ of compressible viscous gas, a complete classification in terms of $(v_{\pm},u_{\pm})$ for the impermeable wall problem, the inflow problem, and the outflow problem is given by Matsumura in \cite{M-MAA-2001}. For the rigorous mathematical justification of this classification, some results have been obtained which can be summarized as in the following:
\begin{itemize}
\item For the impermeable wall problem, to describe its large time behaviors, it is unnecessary to introduce the boundary layer solution and the nonlinear stability of the viscous shock wave and the rarefaction wave are well-understood, cf. \cite{M-Mei}, \cite{M-N2000}. It is worth to pointing out that although the nonlinear stability result for the viscous shock wave in \cite{M-Mei} is obtained only for small initial perturbation, the corresponding result in \cite{M-N2000} for the rarefaction wave holds for any large initial perturbation;

\item For the outflow problem, Kawashima, Nishibata, and Zhu~\cite{Kawashima-Nishibata-Zhu} and Kawashima and Zhu~\cite{Kawahima-Zhu} showed that the boundary layer solution together with the superposition of the boundary layer solution and the rarefaction wave are asymptotically nonlinear stable under small initial perturbation, while Nakamura, Nishibata and Yuge~\cite{Nakamura-Nishibata-Yuge} investigated the convergence rate toward the boundary layer solution. Recently, Huang and Qin~\cite{H-Q} show that not only the boundary layer solution but also the superposition of the boundary layer solution and the rarefaction wave are still stable under large initial perturbation and improve the works of~\cite{Kawashima-Nishibata-Zhu} and~\cite{Kawahima-Zhu};
\item For the inflow problem~(\ref{LagrangeI}), Matsumura and Nishihara~\cite{M-N2001} established the asymptotic stability of the boundary layer solution and the superposition of the boundary layer solution and the rarefaction wave when $(v_-,u_-)\in\Omega_{sub}$ together with the assumption that the initial perturbation is small. Shi~\cite{Shi} studied the rarefaction wave case when $(v_-,u_-)\in\Omega_{super}$ under small initial perturbation. Huang, Matsumura and Shi~\cite{H-M-S} demonstrated the stability of the viscous shock wave and the boundary layer solution to the inflow problem (\ref{LagrangeI}), also under small initial perturbation.
\end{itemize}

It is worth to pointing pout that for the the impermeable wall problem and the outflow problem, the corresponding stability results on the boundary layer solution, the rarefaction wave, and/or their superposition hold true even for certain class of large initial perturbation. Thus a problem of interest is how about the case for the inflow problem, that is, do similar stability results on the boundary layer solution and the rarefaction wave hold for the inflow problem? The main purpose of this paper is devoted to this problem. More precisely, what we are interested in this paper is to consider the following two cases concerning the boundary layer solution and the rarefaction wave for the inflow problem \eqref{EulerI}:
\begin{itemize}
\item[{\bf Case I:}] $(v_-,u_-)\in \Omega_{sub}$ and $(v_+,u_+)\in BL_+(v_-,u_-)\cup BL_-(v_-,u_-)$.
Then the time-asymptotic state of the solutions to the inflow problem \eqref{EulerI} is described by the boundary layer solution $(V,U)(x-s_-t)$ which connects $(v_-,u_-)$ with $(v_+,u_+)$,
where
  $$BL_+(v_-,u_-):=\{(v,u)\in BL(v_-,u_-);\ v_-<v\leq v_* \}$$
  and
  \begin{equation*}
    BL_-(v_-,u_-):=\{(v,u)\in BL(v_-,u_-);\ 0<v<v_-\}.
  \end{equation*}
Here,
$(v_*,u_*)$ is the intersection point of $BL(v_-,u_-)$ and $\Gamma_{trans}$, i.e.,
\begin{equation}
  \label{vstar}
  \sqrt{-p'(v_*)}=\frac{u_-}{v_-},\quad u_*=\frac{u_-}{v_-}v_*.
\end{equation}
The boundary layer solution $(V,U)(x-s_-t)$ will be explained in the next section.
\item[{\bf Case II:}] $(v_-,u_-)\in \Omega_{super}$ and $(v_+,u_+)\in R_1(v_-,u_-)$ (or $(v_+,u_+)\in R_2(v_-,u_-)$).
Then the time-asymptotic state of the solutions to the inflow problem (\ref{EulerI}) is given by the 1-rarefaction wave $(v_1^r,u_1^r)(x/t)$ (or the 2-rarefaction wave $(v_2^r,u_2^r)(x/t)$)
connecting $(v_-,u_-)$ with $(v_+,u_+)$.
\end{itemize}
What we want to show is on the nonlinear stability of both the boundary layer solution and the rarefaction wave for a class of large initial perturbation which can allow the initial density to have large oscillation, which improve the works
of Matsumura and Nishihara~\cite{M-N2001} and Shi~\cite{Shi}.
The precise statements of our main results will be given in Theorems 1-3 below.

The present paper is organized as follows. After stating the notations, in section 2,
we introduce some properties of the boundary layer solution
and the smooth rarefaction wave, and then state the main results.
In section 3, we establish a priori estimates and then prove the
stability of boundary layer by making use of Kanel's technique.
In section 4, the stability of rarefaction wave under large initial perturbation
 will be treated by the similar method.\\

\noindent\emph{Notations.}
Throughout this paper, $c$ and $C$ denote some positive constant (generally large), $\epsilon$, $\lambda$ stand for some positive constants (generally small), and $C(\cdot,\cdot)$ denotes for some generic positive constant depending only on the quantities listed in the parenthesis. Notice that all the constants $c$, $C$, $C(\cdot,\cdot)$, $\epsilon$, and $\lambda$ may take different values in different places. $A\lesssim B$ means that there is a generic constant $C>0$ such that $A\leq CB$ and $A\sim B$ means $A\lesssim B$ and $B\lesssim A$.
For function spaces,~$L^p(\mathbb{R}_+)(1\leq p\leq \infty)$~denotes
the usual Lebesgue space on~$\mathbb{R}_+$~with norm~$\|{\cdot}\|_{L^p}$
and~$H^k(\mathbb{R}_+)$~the usual Sobolev space in the $L^2$ sense with norm~$\|\cdot\|_p$.
We note~$\|\cdot\|=\|\cdot\|_{L^2}$~for simplicity.
Finally, We denote by $C^k(I; H^p)$ the space of $k$-times continuously
differentiable functions on the interval $I$ with values in
$H^p(\mathbb{R}_+)$ and~$L^2(I; H^p)$ the space of~$L^2$-functions
on $I$ with values in~$H^p(\mathbb{R}_+)$.

%Throughout this paper, the notation $A\lesssim B$ will mean that
%$A\leq CB$ holds uniformly over the range of relevant papameters
%which are present in the inequality for some positive constant
%independent of $\epsilon$ or $\delta$.

\section{Preliminaries and Main Results}

\subsection{Boundary Layer Solution}

First we recall some properties about the boundary layer solution.
In~\cite{M-N2001}, it is shown that
if $(v_-,u_-)\in \Omega_{sub}$ and $(v_+,u_+)\in BL_+(v_-,u_-)\cup BL_-(v_-,u_-)$,
the solution to (\ref{LagrangeI}) tends to a boundary layer $(V,U)(\xi)\equiv(V,U)(x-s_-t)$
which is defined by
\begin{equation} \label{BL}
  \begin{cases}
    -s_-V_{\xi}-U_{\xi}=0,\qquad \xi>0,\\[2mm]
    -s_-U_{\xi}+p(V)_{\xi}=\mu\left(\frac{U_{\xi}}{V}\right)_{\xi},\\[2mm]
    (V,U)(0)=(v_-,u_-),\quad (V,U)(\infty)=(v_+,u_+).
  \end{cases}
\end{equation}
The strength of the boundary layer solution $(V,U)(\xi)$ is measured by
\begin{equation}
  \delta:=|u_+-u_-|.
\end{equation}
The existence and the properties of the boundary layer solution $(V,U)(\xi)$ are given in the following lemma.

\begin{Lemma}[cf. \cite{M-N2001}]
If $(v_-,u_-)\in\Omega_{sub}$ and $(v_+,u_+)\in BL_+(v_-,u_-)\cup BL_-(v_-,u_-)$,
there exists a unique solution $(V,U)(\xi)$ to (\ref{BL}) satisfying
for $k=0,1,2,$
\begin{equation} \label{BLdecay}
  \begin{aligned}
    \left|\partial_{\xi}^k(V-v_+,U-u_+)(\xi)\right|&\leq C \delta e^{-c\xi}~\! \qquad \qquad\qquad\textrm{if}\ v_+<v_*,\\[2mm]
    \left|\partial_{\xi}^k(V-v_+,U-u_+)(\xi)\right|&\leq C \delta^{1+k} (1+\delta\xi)^{-1-k} \quad \textrm{if}\ v_+=v_*,
  \end{aligned}
\end{equation}
and
\begin{equation}
  \label{BLdecay1}
  \left|(V_{\xi},V_{\xi\xi},U_{\xi\xi})\right|\leq C |U_{\xi}|,
\end{equation}
the constants $c$ and $C$ depending only on $(v_-,u_-)$.
Furthermore, the boundary layer $(V,U)(\xi)$ is monotonic, that is,
$V_{\xi}\gtrless 0$ and $U_{\xi}\gtrless 0$ if $u_+\gtrless u_-$.
\end{Lemma}

The first aim of this paper is to show the boundary layer solution
obtained in Lemma 2.1 is still stable under some large initial
perturbation.
Defining the perturbation $(\phi,\psi)(t,\xi)$ by
\begin{equation}
(\phi,\psi)(t,\xi)=(v,u)(t,\xi)-(V,U)(\xi),
\end{equation}
we get from (\ref{LagrangeI}) and (\ref{BL}) that $(\phi,\psi)$ satisfies
\begin{equation} \label{BLperturbation}
  \begin{cases}
   {\phi}_t-s_{-}{\phi}_{\xi}-{\psi}_{\xi}=0, \quad  \xi>0,\ t>0,\\[2mm]
   {\psi}_t-s_{-}{\psi}_{\xi}+(p(V+\phi)-p(V))_{\xi}=
   \mu\left(\frac{U_{\xi}+{\psi}_{\xi}}{V+\phi}-\frac{U_{\xi}}{V}\right)_{\xi}, \\[2mm]
   (\phi,\psi)|_{\xi=0}=(0,0),\\[2mm]
   (\phi,\psi)|_{t=0}=({\phi}_0,{\psi}_0):=(v_0-V,u_0-U).
  \end{cases}
\end{equation}
The solution space is
\begin{equation*}
  \begin{aligned}
    X_{m,M}(0,T)
    =\bigg\{(\phi,\psi)\in C([0,T];H_0^1);\ {\phi}_{\xi}\in L^2(0,T;L^2),{\psi}_{\xi}\in L^2(0,T;H^1),&\\[2mm]
\sup_{[0,T]\times\mathbb{R}_+}(V+\phi)(t,\xi)\leq M,\inf_{[0,T]\times\mathbb{R}_+}(V+\phi)(t,\xi)\geq m&\bigg\}.
  \end{aligned}
\end{equation*}
Then the time-local existence of the solution $(\phi,\psi)(t,\xi)$ to (\ref{BLperturbation}) is
quoted in the next lemma.
\begin{Lemma}
\label{locals} (\cite{M-N2001})
Let $({\phi}_0,{\psi}_0)$ be in $H_0^1(\mathbb{R}_{+})$.
If $\sup_{{\mathbb{R}_+}}(V+\phi_0)\leq M$ and
$\inf_{{\mathbb{R}_+}}(V+\phi_0)\geq m$,
then there exists
$t_0 >0$
depending only on $m$, $M$ and $\|({\phi}_0,{\psi}_{0})\|_1$
such that (\ref{BLperturbation}) has a unique solution
$(\phi,\psi)\in X_{m/2,2M}(0,t_0)$ satisfying
\begin{equation} \label{11}
  \|(\phi,{\psi})(t)\|_1 \leq 2\|(\phi_0,\psi_0)\|_1
\end{equation}
and
\begin{equation} \label{Equality}
  s_-\phi_{\xi}(t,0)+\psi_{\xi}(t,0)=0
\end{equation}
for each $0\leq t\leq t_{0}$.
\end{Lemma}

Under the above preparation, we give the following stability result of the boundary layer solution
$(V,U)(\xi)$ which is increasing.
\begin{Theorem}
Assume that $(v_-,u_-)\in \Omega_{sub}$ and $(v_+,u_+)\in BL_{+}(v_-,u_-)$.
Let $({\phi}_0,{\psi}_0)\in H_0^1(\mathbb{R}_{+})$ satisfy
\begin{equation} \label{ic1}
  {\|(\phi_0,\psi_0)\|}\leq C \epsilon^{\alpha},\ \|(\phi_{0\xi},\psi_{0\xi})\|\leq C \left(\epsilon^{-\beta}+1\right),\
  \textrm{and}\
  C^{-1}\epsilon^{l}\leq V+\phi_0\leq C\epsilon^{-l},
\end{equation}
where $C$ is a positive constant independent of $\epsilon$.
If the indices $l\geq 0,\ \alpha\ \textrm{and}\ \beta$ satisfy
\begin{equation} \label{indexBL+}
  \begin{cases}
    \alpha>\frac{\gamma+2}{2}l,\quad
    \alpha+\beta>(\gamma+1)l,\quad \beta\geq \frac{\gamma-1}{2}l,\\[2mm]
    \alpha-\beta\leq \frac{\gamma+3}{2}l,\\[2mm]
        \beta-\alpha< \min\left \{ \frac{2(\gamma-1)}{\gamma+1}\alpha-\frac{3\gamma^2+4\gamma+1}{2\gamma+2}l,
    \frac{3(\gamma-1)}{\gamma^2+1}\alpha-\frac{\gamma^2+6\gamma+1}{2(\gamma^2+1)}\gamma l\right\},\\[2mm]
    \beta-\alpha<\min\left\{\alpha-(\gamma+2)l, \frac{2\alpha}{\gamma}-\frac{\gamma^2+5\gamma+2}{2\gamma}l\right\},
  \end{cases}
\end{equation}
then there exists a suitably small $\epsilon_0>0$ such that if $\epsilon\leq\epsilon_0$,
(\ref{LagrangeI}) has a unique solution $(v,u)$ satisfying
$
  (v-V,u-U)\in C([0,\infty);H_0^1),
$
where $(V,U)$ is defined by (\ref{BL}).

Furthermore, it holds
\begin{equation} \label{largetime}
  \sup_{x\geq s_-t}|(v,u)(t,x)-(V,U)(x-s_-t)|
  \to 0, \ \textrm{as}\ t\to\infty.
\end{equation}
\end{Theorem}

\begin{Remark} Several remarks concerning Theorem 1 are listed below:
\begin{itemize}
\item Here in Theorem 1 the strength of the boundary layer solution is not assumed to be small and thus we can show the nonlinear stability of strong increasing boundary layer solution.
 Since $l=0$ and $0<\alpha\leq \beta<\alpha+\min\left\{1,\frac{2}{\gamma},
 \frac{2(\gamma-1)}{\gamma+1},\frac{3(\gamma-1)}{\gamma^2+1}\right\}\alpha$ imply (\ref{indexBL+}) and in this case the oscillation of the initial density can be large.
\item It is easy to construct some initial perturbation $(\phi_0(x),\psi_0(x))$ satisfying the conditions listed in Theorem 1. In fact for each function $(f(x), g(x)\in H^1(\mathbb{R}_+)$ and each $\alpha, \beta$ satisfying the conditions listed in Theorem 1, if we set
$$
\phi(x)=\epsilon^{\frac{\alpha+\beta}{2}}f\left(\epsilon^{\beta-\alpha}x\right),\quad
\psi(x)=\epsilon^{\frac{\alpha+\beta}{2}}g\left(\epsilon^{\beta-\alpha}x\right),
$$
one can verify that such a $(\phi(x),\psi(x))$ satisfies all the condiitons listed in Theorem 1.
\end{itemize}
\end{Remark}

For the case when the boundary layer solution is decreasing, we have
\begin{Theorem}
Assume that $(v_-,u_-)\in\Omega_{sub}$ and $(v_+,u_+)\in BL_-(v_-,u_-)$.
Let $({\phi}_0,{\psi}_0)\in H_0^1(\mathbb{R}_{+})$ satisfy
\begin{equation}\label{ic2}
  {\|(\phi_0,\psi_0)\|}\leq C \delta^{\alpha},\ \|(\phi_{0\xi},\psi_{0\xi})\|\leq C (\delta^{-\beta}+1)\
  \textrm{and}\
  C^{-1}\delta^{l}\leq V+\phi_0\leq C\delta^{-l},
\end{equation}
where $C$ is a positive constant independent of $\delta$.
If the indices $l\geq 0,\ \alpha\ \textrm{and}\ \beta$ satisfy
\begin{equation} \label{indexBL-}
  \begin{cases}
       \alpha>\frac{\gamma+2}{2}l,\quad
       \gamma l<\alpha+\beta<\frac{1}{2}-\frac{\gamma+5}{2}{l},\\[2mm]
       \alpha-\beta\leq \frac{\gamma+3}{2}l,\\[2mm]
       \beta-\alpha<
       \min\left\{\alpha-(\gamma+2)l,
       \frac{(\gamma-1)(1-4\alpha)}{2\gamma^2+6\gamma-2}
       -\frac{\gamma^3+4\gamma^2+8\gamma-1}{2\gamma^2+6\gamma-2}l\right\}
  \end{cases}
\end{equation}
then there exists $\delta_0>0$ suitably small such that if $\delta\leq\delta_0$,
(\ref{LagrangeI}) has a unique solution $(v,u)$ satisfying
$  (v-V,u-U)\in C([0,\infty);H_0^1) $
and (\ref{largetime}),
where $(V,U)$ is defined by (\ref{BL}).
\end{Theorem}
\begin{Remark}
   $l=0$, $\alpha+\beta<\frac{1}{2}$ and
   $0<\alpha\leq \beta<\alpha+\min\left\{\alpha,\frac{(\gamma-1)(1-4\alpha)}{2\gamma^2+6\gamma-2}\right\}$
   imply (\ref{indexBL-}) and also in such a case the oscillation of the initial density can be large.
\end{Remark}

\subsection{Rarefaction wave}
We only consider the case when $(v_-,u_-)\in\Omega_{super}$ and $(v_+,u_+)\in R_1(v_-,u_-)$
because the study of $(v_+,u_+)\in R_2(v_-,u_-)$ is similar to that of $(v_+,u_+)\in R_1(v_-,u_-)$.
Then the solution to the inflow problem (\ref{LagrangeI}) is expected to tend to
the 1-rarefaction wave connecting $(v_-,u_-)$ and $(v_+,u_+)$.

Since the rarefaction wave is only Lipschitz continuous,
we shall construct a smooth approximation for the rarefaction wave as follows.
First consider the Riemann problem for Burgers' equation:
\begin{equation} \label{BurgerInviscid}
  \begin{cases}
    w_t+ww_x=0,\\
    w(0,x)=w_0^R(x)=
    \begin{cases}
      w_-,\quad x<0,\\
      w_+,\quad x>0,
    \end{cases}
  \end{cases}
\end{equation}
where $w_{\pm}=\lambda_1(v_{\pm})$.
It is obvious that  $w_-<w_+$.
Then it is well known that
(\ref{BurgerInviscid}) has a continuous weak solution $w^r(x/t)$ given by
\begin{equation}
  w^r({x}/{t})=
  \begin{cases}
  w_-,\quad x<w_-t,\\
  {x}/{t},\quad w_-t\leq x \leq w_+t,\\
  w_+,\quad x>w_+t.
  \end{cases}
\end{equation}

Define $(v^r,u^r)(x/t)$ by
\begin{equation} \label{rarefaction}
\begin{cases}
   v^r=\lambda_1^{-1}(w^r),\\
    u^{r}=u_--\int_{v_-}^{v^r}\lambda_1(s)\mathrm{d}s.
\end{cases}
\end{equation}
Then by a simple calculation,
$(v^r,u^r)(x/t)$ satisfies the following Riemann problem of Euler equations, i.e.,
\begin{equation} \label{RP}
  \begin{cases}
    v_t-u_x=0,  \\
    u_t+p(v)_x=0, \\
    (v,u)(0,x)=
    \begin{cases}
      (v_-,u_-),\quad x<0,\\
      (v_+,u_+),\quad x>0.
    \end{cases}
  \end{cases}
\end{equation}

To construct the smooth approximate rarefaction wave $(\tilde{V},\tilde{U})(t,x)$,
we consider the following Cauchy problem for Burgers' equation:
\begin{equation} \label{Burger}
  \begin{cases}
    w_t+ww_x=0,\\
    w(0,x)=w_0(x)=
    \begin{cases}
      w_-,\quad x<0,\\
      w_-+C_q \delta_r \int_0^{{\epsilon}x}y^q e^{-y}\mathrm{d}y,\quad x\geq 0,
    \end{cases}
  \end{cases}
\end{equation}
where $\delta_r=w_+-w_-$, $q \geq 10$ is some constant, $C_q$ is a constant such that
 $C_q\delta_r\int_0^{\infty}y^q e^{-y}\mathrm{d}y = 1$,
$\epsilon \leq 1$ is a positive constant to be determined later.
Then we have
\begin{Lemma}(See \cite{H-Q})
  The problem (\ref{Burger}) has a unique smooth solution $w(t,x)$ satisfying
  \begin{itemize}
    \item[(i)] $w_-\leq w(t,x)<w_+,\ w_x\geq 0$;
    \item[(ii)] for each $1\leq p\leq \infty$, there exists a constant $C$, depending only on
    $p$ and $q$ such that for $t\geq 0$,
    \begin{equation*}
      \begin{aligned}
        \|w_x(t)\|_{L^p}&\leq C\min\left\{
        \delta_r{\epsilon}^{1-1/p},\delta_r^{1/p}t^{-1+1/p}\right\},\\
        \|w_{xx}(t)\|_{L^p}&\leq C\min\left\{\delta_r{\epsilon}^{2-1/p},
        \delta_r^{1/q}t^{-1+1/q}\right\};
      \end{aligned}
    \end{equation*}
    \item[(iii)] when $x\leq w_-t$, $\partial_x^k(w(t,x)-w_-)=0$ for $k=0,1,2$;
    \item[(iv)] $\sup_{x\in\mathbb{R}}|w(1+t,x)-w^r(x/t)|\to 0,$ as $t\to\infty$.
  \end{itemize}
\end{Lemma}

We recall (\ref{rarefaction}) and so define
the smooth approximation $(\tilde{V},\tilde{U})(t,x)$ to $(v^r,u^r)(x/t)$ by
\begin{equation}
\begin{cases}
    \tilde{V}(t,x)=\lambda_1^{-1}(w(1+t,x)),\\[2mm]
     \tilde{U}(t,x)=u_--\int_{v_-}^{\tilde{V}(t,x)}\lambda_1(s)\mathrm{d}s,
\end{cases}
\end{equation}
which satisfies
\begin{equation}
  \begin{cases}
    \tilde{V}_t-\tilde{U}_x=0,\quad x\in\mathbb{R},\ t>0, \\[2mm]
    \tilde{U}_t+p(\tilde{V})_x=0.
  \end{cases}
\end{equation}
Then,
define
\begin{equation} \label{RWVU}
  (V,U)(t,\xi)=(\tilde{V},\tilde{U})(t,x)|_{x\geq s_-t},\quad \xi=x-s_-t,
\end{equation}
and the following lemma holds.
\begin{Lemma}
  $(V,U)(t,\xi)$ satisfies
  \begin{itemize}
    \item[(i)] $U_{\xi}\geq 0$;
    \item[(ii)] for each $1 \leq p \leq\infty$, there exists a constant $C$, depending only on
    $p$, $q$ and $v_{\pm}$, such that for $t\geq 0$,
    \begin{equation}\label{Uxip}
    \|(V_{\xi},U_{\xi})(t)\|_{L^p}\leq C\min\left\{
        {\epsilon}^{1-1/p},(1+t)^{-1+1/p}\right\},
        \end{equation}
    \begin{equation} \label{Uxixip}
      \begin{aligned}
        \|(V_{\xi\xi},U_{\xi\xi},({U_{\xi}}/{V})_{\xi})(t)\|_{L^p}
        \leq
        C\min\Big\{{\epsilon}^{2-{1}/{p}},(1+t)^{-1+{1}/{q}} \Big\};
      \end{aligned}
    \end{equation}
    \item[(iii)] $\sup_{\xi\in\mathbb{R}_+}|(V,U)(t,\xi)-(v^r,u^r)(\frac{\xi+s_-t}{t})|\to 0,$ as $t\to\infty$;
  \end{itemize}
  and also
  \begin{equation} \label{RW}
  \begin{cases}
    {V}_t-s_-V_{\xi}-{U}_{\xi}=0,\quad \xi>0,\ t>0, \\[2mm]
    {U}_t-s_-U_{\xi}+p({V})_{\xi}=0,\\[2mm]
    (V,U)(t,0)=(v_-,u_-),\quad (V,U)(t,\infty)=(v_+,u_+).
  \end{cases}
\end{equation}
\end{Lemma}

Put the perturbation $(\phi,\psi)(t,\xi)$ by
\begin{equation}
(\phi,\psi)(t,\xi)=(v,u)(t,\xi)-(V,U)(\xi),
\end{equation}
then the reformulated problem is
\begin{equation} \label{RWperturbation}
  \begin{cases}
   {\phi}_t-s_{-}{\phi}_{\xi}-{\psi}_{\xi}=0, \quad  \xi>0,\ t>0,\\[2mm]
   {\psi}_t-s_{-}{\psi}_{\xi}+(p(V+\phi)-p(V))_{\xi}
   =\mu\left(\frac{U_{\xi}+{\psi}_{\xi}}{V+\phi}\right)_{\xi}, \\[2mm]
   (\phi,\psi)|_{\xi=0}=(0,0),\\[2mm]
   (\phi,\psi)|_{t=0}=({\phi}_0,{\psi}_0)(\xi):=(v_0-V,u_0-U)(\xi),
  \end{cases}
\end{equation}
from (\ref{LagrangeI}) and (\ref{RW}).
Then the time-local existence of the solution $(\phi,\psi)(t,\xi)$ to (\ref{RWperturbation}) is
quoted in the next lemma.
\begin{Lemma}
\label{localsRW}
Let $({\phi}_0,{\psi}_0)$ be in $H_0^1(\mathbb{R}_{+})$.
If $\sup_{{\mathbb{R}_+}}(V+\phi_0)\leq M$ and
$\inf_{{\mathbb{R}_+}}(V+\phi_0)\geq m$,
then there exists
$t_0 >0$
depending only on $m$, $M$ and $\|({\phi}_0,{\psi}_{0})\|_1$
such that (\ref{RWperturbation}) has a unique solution
$(\phi,\psi)\in X_{m/2,2M}(0,t_0)$ satisfying
(\ref{11}) and (\ref{Equality})
for each $0\leq t\leq t_{0}$.
\end{Lemma}

With the above result in hand, for the nonlinear stability of supersonic rarefaction wave, we can get that
\begin{Theorem}
Assume that $(v_-,u_-)\in \Omega_{super}$ and $(v_+,u_+)\in R_1(v_-,u_-)$.
Let $({\phi}_0,{\psi}_0)\in H_0^1(\mathbb{R}_{+})$ satisfy
\begin{equation}
  {\|(\phi_0,\psi_0)\|}_1\leq C (\epsilon^{-\alpha}+1),\ u_-\geq C\epsilon^{-l_0},\
  \textrm{and}\
  C^{-1}\epsilon^{l}\leq V+\phi_0\leq C\epsilon^{-l},
\end{equation}
where $C$ is a positive constant independent of $\epsilon$.
If the indices $l\geq 0,\ \alpha\ \textrm{and}\ l_0$ satisfy
\begin{equation} \label{indexR}
  \begin{cases}
    l\leq \frac{1}{8\gamma-2},\\[2mm]
    l_0>2\alpha+(\gamma+1)l,\\[2mm]
    4\alpha+2(\gamma+1)l\geq \max\{l,(\gamma-1)l\},\\[2mm]
    4\alpha+2(\gamma+1)l<\min\left\{\frac{\gamma-1}{6\gamma^2-6\gamma+2},%\frac{\gamma-1}{\gamma^2+1},
    \frac{2(\gamma-1)}{3\gamma^2+3\gamma+9}\right\},
  \end{cases}
\end{equation}
there exists $\epsilon_0>0$ suitably small such that if $\epsilon\leq\epsilon_0$,
(\ref{LagrangeI}) has a unique solution $(v,u)$ satisfying
$
  (v-V,u-U)\in C([0,\infty);H_0^1),
$
where $(V,U)$ is defined by (\ref{RWVU}).
Furthermore, it holds
\begin{equation} \label{largetimeRW}
  \sup_{x\geq s_-t}|(v,u)(t,x)-(V,U)(t,x-s_-t)|
  \to 0, \ \textrm{as}\ t\to\infty.
\end{equation}
\end{Theorem}

\begin{Remark}
  $l=0$, $0\leq \alpha<\frac{1}{4}\min\left\{\frac{\gamma-1}{6\gamma^2-6\gamma+2},
    \frac{2(\gamma-1)}{3\gamma^2+3\gamma+9}\right\}$ and $l_0>2\alpha$ imply (\ref{indexR}).
\end{Remark}

\subsection{Main Difficulties and Ideas}

To deduce the desired nonlinear stability result for the boundary layer solution, the rarefaction wave, and/or their superposition by the elementary energy method as in \cite{M-N2000}, it is sufficient to deduce certain uniform (with respect to the time variable $t$) energy type estimates on the solutions $(\phi(t, x), \psi(t,x))$ and the main difficulties to do so lie in the following:
\begin{itemize}
\item How to control the possible growth of $(\phi(t,x), \psi(t,x))$ caused by the nonlinearity of the equation \eqref{EulerI}$_1$-\eqref{EulerI}$_2$?
\item How to control the term
$$
\int_0^t \frac{1}{v_-^2}\phi_{\xi}^2(\tau,0)d\tau
  =\int^t_0\frac{1}{u_-^2}\psi_{\xi}^2(\tau,0)d\tau
$$
which is due to the inflow boundary condition \eqref{EulerI}$_3$?
\end{itemize}

The argument employed in \cite{M-N2000} is to use the smallness of $N(T):=\sup_{0\leq t\leq T}\left\{\|(\phi,\psi)(t)\|_1\right\}$ to overcome the above difficulties. One of the key points in such an
argument is that, based on the a priori assumption that $N(T)$ is sufficiently small,
one can deduce a uniform lower and upper positive bounds on the specific volume
$v(t,x)$. With such a bound on $v(t,x)$ in hand, one can thus deduce certain a priori
$H^1({\mathbb{R}_+})$ energy type estimates on $(\phi(t,x), \psi(t,x))$ in terms of the initial perturbation
$(\phi_0(x),\psi_0(x))$. Then combination of the above analysis with the standard continuation argument
yields the corresponding nonlinear stability result. It is worth pointing out that for the case when the strength of the underlying profile is small, for the nonlinear stability result obtained in \cite{M-N2000}, $ {\textrm{Osc}} v(t) := \sup_{x\in{\mathbb{R}_+}} v(t,x)-\inf_{x\in{\mathbb{R}_+}} v(t,x),$ the oscillation of the specific volume $v(t,x)$, should be sufficiently small also for all $t\in{\mathbb{R}_+}$.

What we are interested in this paper is to deduce the corresponding nonlinear stability results for the two cases listed in the introduction for a class of initial perturbation which can allow the initial density to have large oscillation, the argument used in \cite{M-N2000} can not used any longer. Our main ideas to yield the desired nonlinear stability results are the following:
\begin{itemize}
\item For the nonlinear stability of the boundary layer solution listed in Case I of the introduction, our main observation is that for the case when the underlying boundary layer solution is increasing, the basic energy estimate, cf. the estimate \eqref{1stEstimateBL+}, tells us that for each $t\in[0,T]$ the instant energy $\mathcal{E}(t)=\int_0^{+\infty}\left(\Phi(v(t,x),V(x))+\frac 12\psi^2(t,x)\right)dx$
    is bounded by the initial energy $\mathcal{E}(0)$ and thus one may use the smallness of $\mathcal{E}(0)$ to overcome the two difficulties mentioned above. Since our main purpose is to get a nonlinear stability result for which the oscillation of the specific volume $v(t,x)$ can be large, we need to deduce a precise estimates on $v(t,x)$ in terms of $\mathcal{E}(0)$ so that the whole analysis can be carried out. It is worth to emphasizing that Kanel's argument \cite{Kanel'} plays an important role in this step and it was to guarantee that the whole analysis to be carried out smoothly that we need to ask the parameters $\alpha, \beta,$ and $l$ to satisfies the conditions listed in Theorem 1.

\item For the case when the boundary layer solution is decreasing, the analysis can be adopted directly since in such a case the basic energy estimate is not self-contained. Even so, if the strength of the boundary layer solution is small, one can use the smallness of both the initial energy and the strength of the boundary layer solution to yield a nonlinear stability result similar to that of the case when the boundary layer solution is increasing.

\item For the nonlinear stability of supersonic rarefaction wave corresponding to the Case II listed in the introduction, our main idea is to use the largeness of $u_-$ to deal with the two difficulties mentioned above. In such a case, we do not ask the initial energy to be small and thus such a result holds for a class of large initial perturbation.
\end{itemize}

\section{Stability of the Boundary Layer Solution}
\subsection{Proof of Theorem 1}
In this subsection,
we first assume that $(v_-,u_-)\in \Omega_{sub}$, $(v_+,u_+)\in BL_{+}(v_-,u_-)$ and
the problem (\ref{BLperturbation}) has a solution
$(\phi,\psi)\in X_{1/m,M}(0,T)$ satisfying (\ref{Equality}) for some $T>0$ and each $0\leq t\leq T$.
We also simply write $c$ and $C$ as positive constants
independent of $T,\ m,\ M$ and $\epsilon$. Recall that the notation $A\lesssim B$ is used to denote that
$A\leq CB$ holds uniformly for some positive constant independent of $T,\ m,\ M$ and $\epsilon$.
Besides, we will often use the notation $(v, u) = (V +\phi,U +\psi)$
so that $1/m\leq v\leq M$,
though the unknown functions are $\phi$ and $\psi$.
Without loss of generality, we choose $m$ and $M$  such that $m,M\geq 1$.

Now we devote ourselves to the basic energy estimate.
\begin{Lemma}
It holds that for each $0\leq t\leq T$,
\begin{equation} \label{1stEstimateBL+}
\begin{aligned}
  \left\|(\sqrt{\Phi},\psi)(t)\right\|^2
  +\int^t_0\int_0^{\infty}\left[\frac{{\psi}^2_{\xi}}{v}
  +\frac{|U_{\xi}\phi{\psi}_{\xi}|}{v}+|U_{\xi}|\left(p(v)-p(V)-p'(V)\phi\right)\right]\mathrm{d}\xi \mathrm{d}\tau
  \lesssim \left\|(\sqrt{\Phi_0},\psi_0)\right\|^2,
\end{aligned}
\end{equation}
where
\begin{equation}\label{Phi}
  \Phi=\Phi(v,V)=p(V)(v-V)-\int_V^v p(\eta)\mathrm{d}\eta.
\end{equation}
\end{Lemma}

\begin{proof}
Multiplying $(\ref{BLperturbation})_1$ (the first equation of (\ref{BLperturbation}))
and $(\ref{BLperturbation})_2$
by $p(V)-p(v)$ and $\psi$, respectively, and summing these two identities,
we find a divergence from
\begin{equation} \label{BLdivergence1}
  \begin{aligned}
  &\left(\Phi+\frac{1}{2}\psi^2\right)_t+\mu\frac{\psi_{\xi}^2}{v}
  -\mu\frac{U_{\xi}\phi\psi_{\xi}}{vV}+U_{\xi}(p(v)-p(V)-p'(V)\phi)\\[2mm]
  =&\left[s_-\left(\Phi+\frac{1}{2}\psi^2\right)+(p(V)-p(v))\psi
  +\mu\left(\frac{U_{\xi}+{\psi}_{\xi}}{V+\phi}-\frac{U_{\xi}}{V}\right)\psi\right]_{\xi}.
\end{aligned}
\end{equation}
We write
\begin{equation} \label{f}
  p(v)-p(V)-p'(V)\phi=f(v,V)\phi^2.
\end{equation}
Put $X=V/v>0$ and then recall Bernoulli's inequality
\begin{equation*}
  X^{\gamma+1}=X(1+X-1)^{\gamma}\geq X+\gamma X(X-1)
\end{equation*}
to find
\begin{equation} \label{fvV}
  Vvf(v,V)=V^{-\gamma}\frac{X^{\gamma+1}-(\gamma+1)X+\gamma}{(X-1)^2}\geq \gamma V^{-\gamma}.
\end{equation}
Noting that $v_-\leq V\leq v_+$, we have from (\ref{BL}) that
\begin{equation*}
  0\leq \mu U_{\xi}=V\left(s_-^2(V-v_+)+p(V)-p(v_+)\right)\leq Vp(V)=V^{-\gamma+1}.
\end{equation*}
Thus, the discriminant $D$ of
$$\mu\frac{\psi_{\xi}^2}{v}-\mu\frac{U_{\xi}\phi\psi_{\xi}}{vV}+U_{\xi}(p(v)-p(V)-p'(V)\phi)$$
satisfies
\begin{equation} \label{D}
  D=\frac{\mu U_{\xi}}{V^2vf(v,V)}-4\leq \frac{1}{\gamma}-4<0.
\end{equation}
Therefore, we integrate (\ref{BLdivergence1}) over $(0,t)\times(0,\infty)$ to get (\ref{1stEstimateBL+}).
\end{proof}

Next, following \cite{M-N1992}, we set $\tilde{v}:={v}/{V}$.
Then we have $\Phi(v,V)=V^{-\gamma+1}\tilde{\Phi}(\tilde{v})$
with
\begin{equation} \label{tildeofPhi}
  \tilde{\Phi}(\tilde{v})=\tilde{v}-1+\frac{1}{\gamma-1}(\tilde{v}^{-\gamma+1}-1).
\end{equation}
Equation $(\ref{BLperturbation})_2$ is also written as
\begin{equation} \label{BLperturbation2}
\left(\mu\frac{\tilde{v}_{\xi}}{\tilde{v}}-\psi\right)_t-
s_{-}\left(\mu\frac{\tilde{v}_{\xi}}{\tilde{v}}-\psi\right)_{\xi}
+\frac{\gamma\tilde{v}_{\xi}}{V^{\gamma}{\tilde{v}}^{\gamma+1}}
=\frac{\gamma V_{\xi}}{V^{\gamma+1}}(1-{\tilde{v}}^{-\gamma}).
\end{equation}
Multiplying $(\ref{BLperturbation2})$ by ${\tilde{v}_{\xi}}/{\tilde{v}}$,
we discover
\begin{equation} \label{BLdivergence2}
  \begin{aligned}
    &\left[\frac{\mu}{2}\left(\frac{\tilde{v}_{\xi}}{\tilde{v}}\right)^2
    -\psi\frac{\tilde{v}_{\xi}}{\tilde{v}}\right]_t
     +\frac{\gamma \tilde{v}_{\xi}^2}{v^{\gamma}\tilde{v}^2}\\[2mm]
    &+\left[\psi\frac{\tilde{v}_t}{\tilde{v}}-\frac{\gamma V_{\xi}}{V^{\gamma+1}}\left(\frac{\tilde{v}^{-\gamma}-1}{\gamma}+\ln\tilde{v}\right)
     -\frac{\mu s_-}{2}\left(\frac{\tilde{v}_{\xi}}{\tilde{v}}\right)^2\right]_{\xi}\\[2mm]
    =&~\frac{\psi_{\xi}^2}{v}-\frac{U_{\xi}\phi\psi_{\xi}}{vV}
     \underbrace{-\gamma\frac{VV_{\xi\xi}-(\gamma+1)V_{\xi}^2}{V^{\gamma+2}}
     \left(\frac{\tilde{v}^{-\gamma}-1}{\gamma}+\ln\tilde{v}\right)}_{I_1}.
  \end{aligned}
\end{equation}
We utilize (\ref{BLdecay1}) and the fact that
\begin{equation}
  p(v)-p(V)-p'(V)\phi=V^{-\gamma}(\tilde{v}^{-\gamma}-1+\gamma(\tilde{v}-1))
\end{equation} to find
\begin{equation*}
\begin{aligned}
  |I_1|
  \lesssim |U_{\xi}|(\tilde{v}^{-\gamma}-1+\gamma(\tilde{v}-1))
  \lesssim |U_{\xi}|(p(v)-p(V)-p'(V)\phi).
\end{aligned}
\end{equation*}
Thus, integrating (\ref{BLdivergence2}) over $(0,t)\times(0,\infty)$ yields
\begin{equation} \label{badterm}
  \begin{aligned}
    \left\|\frac{\tilde{v}_{\xi}}{\tilde{v}}(t)\right\|^2+
    \int_0^t\int_0^{\infty}\frac{\tilde{v}_{\xi}^2}{v^{\gamma}\tilde{v}^2}\mathrm{d}\xi \mathrm{d}\tau
    \lesssim\left\|\left(\sqrt{\Phi_0},\psi_0,\frac{\tilde{v}_{0\xi}}{\tilde{v}_0}\right)\right\|^2+%|s_-|
    \int_0^t \left(\frac{\tilde{v}_{\xi}}{\tilde{v}}\right)^2(\tau,0)\mathrm{d}\tau.
  \end{aligned}
\end{equation}
We have to control the final term of (\ref{badterm}),
$\int_0^t(\frac{\tilde{v}_{\xi}}{\tilde{v}})^2(\tau,0)\mathrm{d}\tau$.
We note here that Matsumura-Nishihara \cite{M-N2001} could control it under the smallness assumption that
\begin{equation} \label{smallnessC}
  N(T):=\sup_{0\leq t\leq T}\|(\phi,\psi)(t)\|_1\ll 1,
\end{equation}
while our goal in this paper is to investigate the stability of the boundary layer solution and the rarefaction wave
without such a smallness condition (\ref{smallnessC}).

Since
\begin{equation} \label{log(tildeofv)}
  \frac{\tilde{v}_{\xi}}{\tilde{v}}=\frac{\phi_{\xi}}{v}-\frac{V_{\xi}\phi}{vV},
\end{equation}
we get from $\phi(\tau,0)=0$ and (\ref{Equality}) that
\begin{equation}
  \left(\frac{\tilde{v}_{\xi}}{\tilde{v}}\right)^2(\tau,0)=\frac{1}{v_-^2}\phi_{\xi}^2(\tau,0)
  =\frac{1}{u_-^2}\psi_{\xi}^2(\tau,0),
\end{equation}
which together with (\ref{badterm}) implies the following lemma.
\begin{Lemma} It holds that
\begin{equation} \label{2ndEstimateBL+}
\begin{aligned}
      \left\|\frac{\tilde{v}_{\xi}}{\tilde{v}}(t)\right\|^2+
      \int_0^t\int_0^{\infty}\frac{\tilde{v}_{\xi}^2}{v^{\gamma}\tilde{v}^2}\mathrm{d}\xi \mathrm{d}\tau
    \lesssim\left\|\left(\sqrt{\Phi_0},\psi_0,\frac{\tilde{v}_{0\xi}}{\tilde{v}_0}\right)\right\|^2
    +  \int_0^t \psi_{\xi}^2(\tau,0)\mathrm{d}\tau.
\end{aligned}
\end{equation}
\end{Lemma}

We next estimate the last term
%on the right-hand side
of (\ref{2ndEstimateBL+}).
Apply H\"{o}lder's inequality to find
\begin{equation}
  {\psi}_\xi^2(\tau,0)
=-2\int^{\infty}_0{{\psi}_\xi{\psi}_{\xi\xi}(\tau,\xi)}\mathrm{d}\xi
\leq 2M \left\|\frac{\psi_{\xi}}{\sqrt{v}}(\tau)\right\|\left\|\frac{\psi_{\xi\xi}}{\sqrt{v}}(\tau)\right\|
\end{equation}
and
\begin{equation} \label{boundaryterm1}
\begin{aligned}
   \int_0^t{{\psi}_\xi^2(\tau,0)}\mathrm{d}\tau
  \leq 2M \left[\int_0^t\left\|\frac{\psi_{\xi}}{\sqrt{v}}(\tau)\right\|^2\mathrm{d}\tau\right]^{\frac{1}{2}}
  \left[\int_0^t\left\|\frac{\psi_{\xi\xi}}{\sqrt{v}}(\tau)\right\|^2\mathrm{d}\tau\right]^{\frac{1}{2}}.
\end{aligned}
\end{equation}
Then applying Cauchy's inequality, we have from (\ref{boundaryterm1}) that
for each $a>1$,
\begin{equation} \label{boundaryterm2}
\begin{aligned}
   \int_0^t{{\psi}_\xi^2(\tau,0)}\mathrm{d}\tau
  \leq a\int_0^t\left\|\frac{\psi_{\xi}}{\sqrt{v}}(\tau)\right\|^2\mathrm{d}\tau + 4M^2a^{-1}\int_0^t\left\|\frac{\psi_{\xi\xi}}{\sqrt{v}}(\tau)\right\|^2\mathrm{d}\tau.
\end{aligned}
\end{equation}
Substituting (\ref{boundaryterm2})
%with $a=\epsilon^{-2\alpha-2\beta+(\gamma-1)l}$
into (\ref{2ndEstimateBL+}) and using the basic energy estimate (\ref{1stEstimateBL+}), we deduce
\begin{equation} \label{2ndEstimate'BL+}
\begin{aligned}
  \left\|\frac{\tilde{v}_{\xi}}{\tilde{v}}(t)\right\|^2+
  \int_0^t\int_0^{\infty}\frac{\tilde{v}_{\xi}^2}{v^{\gamma}\tilde{v}^2}\mathrm{d}\xi \mathrm{d}\tau
  \lesssim\left\|\frac{\tilde{v}_{0\xi}}{\tilde{v}_0}\right\|^2
  +a{\left\|(\sqrt{\Phi_0},\psi_0)\right\|}^2+
  M^2a^{-1}\int_0^t\left\|\frac{\psi_{\xi\xi}}{\sqrt{v}}(\tau)\right\|^2
  \mathrm{d}\tau.
\end{aligned}
\end{equation}
It is necessary to estimate the last term of (\ref{2ndEstimate'BL+}).
For this, we multiply $(\ref{BLperturbation})_2$ by $-\psi_{\xi\xi}$ to find
\begin{equation} \label{BLdivergence3}
 \begin{aligned}
   &\left(\frac{1}{2}\psi_{\xi}^2\right)_t
+\left(\frac{s_-}{2}\psi_{\xi}^2-\psi_t\psi_{\xi}\right)_{\xi}+\mu\frac{\psi_{\xi\xi}^2}{v}\\[2mm]
   =~&\underbrace{\mu\psi_{\xi\xi}\frac{V_{\xi}\psi_{\xi}}{v^2}
  +\mu\psi_{\xi\xi}\frac{\phi_{\xi}\psi_{\xi}}{v^2}}_{R_1}
+\underbrace{\psi_{\xi\xi}\phi_{\xi}\left(p'(v)+\mu\frac{U_{\xi}}{v^2}\right)}_{R_2}\\[2mm]
&+\underbrace{\psi_{\xi\xi}\left[(p'(v)-p'(V))V_{\xi}+\mu\frac{U_{\xi\xi}\phi}{vV}
  -\mu\frac{U_{\xi}V_{\xi}\phi}{v^2V}-\mu\frac{U_{\xi}V_{\xi}\phi}{vV^2}\right]}_{R_3},
 \end{aligned}
\end{equation}
%with
%\begin{equation} \label{R}
%\begin{aligned}
%  &R_1=\psi_{\xi\xi}\phi_{\xi}\left[p'(v)+\mu\frac{U_{\xi}}{v^2}\right],\\[2mm]
%  &R_2=\psi_{\xi\xi}\left[(p'(v)-p'(V))V_{\xi}+\mu\frac{U_{\xi\xi}\phi}{vV}
%  -\mu\frac{U_{\xi}V_{\xi}\phi}{v^2V}-\mu\frac{U_{\xi}V_{\xi}\phi}{vV^2}\right],\\[2mm]
%  &R_3=\mu\psi_{\xi\xi}\frac{V_{\xi}\psi_{\xi}}{v^2}
%  +\mu\psi_{\xi\xi}\frac{\phi_{\xi}\psi_{\xi}}{v^2}.
%\end{aligned}
%\end{equation}
Apply Cauchy's inequality and (\ref{log(tildeofv)}) to find
\begin{equation} \label{R1}
  R_1\leq \frac{\mu}{16}\frac{\psi_{\xi\xi}^2}{v}+CV_{\xi}^2\frac{\psi_{\xi}^2}{v^3}
+8\mu\frac{\phi_{\xi}^2\psi_{\xi}^2}{v^3}
\end{equation}
and
\begin{equation} \label{R2}
  \begin{aligned}
    R_2
       \leq& \frac{\mu}{32}\frac{\psi_{\xi\xi}^2}{v}
             +Cv\left(v^{-2\gamma-2}+v^{-4}U_{\xi}^2\right)
             \left(\tilde{v}_{\xi}^2+{V_{\xi}^2\phi^2}\right).
  \end{aligned}
\end{equation}
Noting that
\begin{equation}
  |p'(v)-p'(V)|
  =\int_0^1 p''(\theta v+(1-\theta)V)\mathrm{d}\theta |\phi|
  \lesssim (v^{-\gamma-2}+1)|\phi|,
\end{equation}
we have
\begin{equation} \label{R3}
  \begin{aligned}
    R_3
    \leq& \frac{\mu}{32}\frac{\psi_{\xi\xi}^2}{v}
              +Cv\left[V_{\xi}^2(v^{-2\gamma-4}+1)
              +{v^{-2}}{U_{\xi\xi}^2}+(v^{-4}+v^{-2}){U_{\xi}^2V_{\xi}^2}\right]\phi^2.
  \end{aligned}
\end{equation}
We now estimate the last term of (\ref{R1}). Applying Sobolev's inequality and Cauchy's inequality, we get
\begin{equation}
  \begin{aligned}
    &\int_0^{\infty}\frac{\phi_{\xi}^2\psi_{\xi}^2}{v^3}\mathrm{d}\xi
    \leq \left\|\frac{\phi_{\xi}}{v}\right\|^2\left\|\frac{\psi_{\xi}}{\sqrt{v}}\right\|_{L^{\infty}}^2 \\[2mm]
    \leq~& \left\|\frac{\phi_{\xi}}{v}\right\|^2\left\|\frac{\psi_{\xi}}{\sqrt{v}}\right\|
    \left[\left\|\frac{\psi_{\xi\xi}}{\sqrt{v}}\right\|+\left\|\frac{\phi_{\xi}\psi_{\xi}}{2v^{3/2}}\right\|
    +\left\|\frac{V_{\xi}\psi_{\xi}}{2v^{3/2}}\right\|\right]\\[2mm]
    \leq~& \frac{1}{128}\left\|\frac{\psi_{\xi\xi}}{\sqrt{v}}\right\|^2+
    \frac{1}{2}\left\|\frac{\phi_{\xi}\psi_{\xi}}{v^{3/2}}\right\|^2
    +\frac{1}{2}\left\|\frac{V_{\xi}\psi_{\xi}}{v^{3/2}}\right\|^2+
    C\left\|\frac{\phi_{\xi}}{v}\right\|^4\left\|\frac{\psi_{\xi}}{\sqrt{v}}\right\|^2,
  \end{aligned}
\end{equation}
which implies
\begin{equation} \label{R1_2}
  \int_0^{\infty}\frac{\phi_{\xi}^2\psi_{\xi}^2}{v^3}\mathrm{d}\xi
  \leq \frac{1}{64}\left\|\frac{\psi_{\xi\xi}}{\sqrt{v}}\right\|^2
    +\left\|\frac{V_{\xi}\psi_{\xi}}{v^{3/2}}\right\|^2+C\left\|\frac{\phi_{\xi}}{v}\right\|^4
    \left\|\frac{\psi_{\xi}}{\sqrt{v}}\right\|^2.
\end{equation}
Integrating (\ref{R1}) over $(0,\infty)$ and then plugging (\ref{R1_2}) into this last inequality, we find
\begin{equation} \label{R1'}
  \int_0^{\infty} R_1\mathrm{d}\xi
  \leq \frac{3\mu}{16}\left\|\frac{\psi_{\xi\xi}}{\sqrt{v}}\right\|^2
+C\left\|\frac{V_{\xi}\psi_{\xi}}{v^{3/2}}\right\|^2
  +C\left\|\frac{\phi_{\xi}}{v}\right\|^4\left\|\frac{\psi_{\xi}}{\sqrt{v}}\right\|^2.
\end{equation}
Therefore, we integrate (\ref{BLdivergence3}) over $(0,t)\times(0,\infty)$ and then recall
the above estimates (\ref{R2}), (\ref{R3}), (\ref{R1'}) and $1/m\leq v\leq M$ to conclude
\begin{equation} \label{3rdEstimateBL+}
  \begin{aligned}
    &\|\psi_{\xi}(t)\|^2+\int_0^t|s_-|\psi_{\xi}^2(\tau,0)\mathrm{d}\tau
    +\int_0^t\left\|\frac{\psi_{\xi\xi}}{\sqrt{v}}(\tau)\right\|^2\mathrm{d}\tau\\[2mm]
    \lesssim ~&\|\psi_{0\xi}\|^2
    +\int_0^t\int_0^{\infty}(m^{\gamma-1}+
    M^{\gamma-1}U_{\xi}^2)\frac{\tilde{v}_{\xi}^2}{v^{\gamma}\tilde{v}^2}\mathrm{d}\xi \mathrm{d}\tau\\[2mm]
    &+\int_0^t\int_0^{\infty}\left[J\phi^2+m^2V_{\xi}^2\frac{\psi_{\xi}^2}{v}\right]\mathrm{d}\xi \mathrm{d}\tau
    +\int_0^t\left\|\frac{\phi_{\xi}}{v}\right\|^4\left\|\frac{\psi_{\xi}}{\sqrt{v}}\right\|^2\mathrm{d}\tau
  \end{aligned}
\end{equation}
with
\begin{equation}\label{J}
  J%=(v^{-2\gamma-1}+v^{-3}U_{\xi}^2)\frac{V_{\xi}^2}{V^2}+V_{\xi}^2(v^{-2\gamma-3}+vV^{-2\gamma-4})
  %+\frac{U_{\xi\xi}^2}{vV^2}+\frac{U_{\xi}^2V_{\xi}^2}{v^3V^4}(v^{2}+V^{2}).
  =(v^{-2\gamma-3}+v)V_{\xi}^2+v^{-1}{U_{\xi\xi}^2}+(v^{-1}+v^{-3}){U_{\xi}^2V_{\xi}^2}.
\end{equation}
We use (\ref{BLdecay1}) and (\ref{fvV}) to deduce
\begin{equation}
    J\lesssim(v^{-2\gamma-3}+v)U_{\xi}^2
  \lesssim(m^{2\gamma+2}+M^2)|U_{\xi}|f(v,V).
\end{equation}
Hence,
we obtain from (\ref{1stEstimateBL+}) and (\ref{f}) that
\begin{equation} \label{3.1.0}
\begin{aligned}
       \int_0^t\int_0^{\infty}\left[J\phi^2+m^2V_{\xi}^2\frac{\psi_{\xi}^2}{v}\right]\mathrm{d}\xi \mathrm{d}\tau
  \lesssim(m^{2\gamma+2}+M^2) {\left\|(\sqrt{\Phi_0},\psi_0)\right\|}^2
\end{aligned}
\end{equation}
and
\begin{equation} \label{3.1.1}
  \int_0^t\left\|\frac{\phi_{\xi}}{v}\right\|^4\left\|\frac{\psi_{\xi}}{\sqrt{v}}\right\|^2\mathrm{d}\tau
  \lesssim{\left\|(\sqrt{\Phi_0},\psi_0)\right\|}^2\sup_{0\leq\tau\leq t}\left\|\frac{\phi_{\xi}}{v}(\tau)\right\|^4.
\end{equation}
Plugging estimates (\ref{3.1.0}) and (\ref{3.1.1}) into (\ref{3rdEstimateBL+}), we have
\begin{equation} \label{3rdEstBL+}
  \begin{aligned}
    &\|\psi_{\xi}(t)\|^2+\int_0^t|s_-|\psi_{\xi}^2(\tau,0)\mathrm{d}\tau
    +\int_0^t\left\|\frac{\psi_{\xi\xi}}{\sqrt{v}}(\tau)\right\|^2\mathrm{d}\tau\\[2mm]
    \lesssim ~&\|\psi_{0\xi}\|^2+(m^{\gamma-1}+M^{\gamma-1})
\int_0^t\int_0^{\infty}\frac{\tilde{v}_{\xi}^2}{v^{\gamma}\tilde{v}^2}\mathrm{d}\xi \mathrm{d}\tau
    \\[2mm]
        &+(m^{2\gamma+2}+M^2)
 {\left\|(\sqrt{\Phi_0},\psi_0)\right\|}^2+{\left\|(\sqrt{\Phi_0},\psi_0)\right\|}^2
\sup_{0\leq\tau\leq t}\left\|\frac{\phi_{\xi}}{v}(\tau)\right\|^4. \end{aligned}
\end{equation}
Since
\begin{equation}
  \Phi(v,V)
  =-\int_0^1\int_0^1\theta_1 p'(\theta_1\theta_2 v+(1-\theta_1\theta_2)V)\mathrm{d}\theta_1\mathrm{d}\theta_2 \phi^2,
\end{equation}
we have
\begin{equation} \label{Phiphi}
  C^{-1}M^{-\gamma-1}\phi^2\leq \Phi(v,V)\leq C m^{\gamma+1}\phi^2.
\end{equation}
Then we deduce from (\ref{log(tildeofv)}) that
\begin{equation} \label{3.1.2}
  \begin{aligned}
    \left\|\frac{\phi_{\xi}}{v}(\tau)\right\|^2
    \leq C\|V_{\xi}(\tau)\|_{L^{\infty}}^2m^2 M^{\gamma+1}\left\|\sqrt{\Phi}(\tau)\right\|^2+
    \left\|\frac{\tilde{v}_{\xi}}{\tilde{v}}(\tau)\right\|^2.
  \end{aligned}
\end{equation}
Hence we have from (\ref{1stEstimateBL+}), (\ref{2ndEstimateBL+}) and (\ref{boundaryterm1}) that
\begin{equation} \label{3.1.3}
  \begin{aligned}
    \left\|\frac{\phi_{\xi}}{v}(\tau)\right\|^4
\lesssim ~&m^4M^{2\gamma+2}{\left\|(\sqrt{\Phi_0},\psi_0)\right\|}^4+
    \left\|\frac{\tilde{v}_{0\xi}}{\tilde{v}_0}\right\|^4
           +\left[\int_0^t \psi_{\xi}^2(\tau,0)\mathrm{d}\tau\right]^2\\[2mm]
    \lesssim ~&m^4M^{2\gamma+2}{\left\|(\sqrt{\Phi_0},\psi_0)\right\|}^4+
    \left\|\frac{\tilde{v}_{0\xi}}{\tilde{v}_0}\right\|^4
           +M^2{\left\|(\sqrt{\Phi_0},\psi_0)\right\|}^2
           \int_0^t\left\|\frac{\psi_{\xi\xi}}{\sqrt{v}}(\tau)\right\|^2\mathrm{d}\tau.
 \end{aligned}
 \end{equation}
Substituting (\ref{3.1.3}) into (\ref{3rdEstBL+}), we find a positive constant $c$ such
that if
\begin{equation}
    M^2{\left\|(\sqrt{\Phi_0},\psi_0)\right\|}^4\leq c,
\end{equation}
then
\begin{equation} \label{3rdEBL+}
  \begin{aligned}
    &\|\psi_{\xi}(t)\|^2+\int_0^t|s_-|\psi_{\xi}^2(\tau,0)\mathrm{d}\tau
    +\int_0^t\left\|\frac{\psi_{\xi\xi}}{\sqrt{v}}(\tau)\right\|^2\mathrm{d}\tau\\[2mm]
%    \lesssim ~&\|\psi_{0\xi}\|^2+(m^{\gamma-1}+M^{\gamma-1})
%\int_0^t\int_0^{\infty}\frac{\tilde{v}_{\xi}^2}{v^{\gamma}\tilde{v}^2}\mathrm{d}\xi \mathrm{d}\tau
%    \\[2mm]
%        &+\left(m^{2\gamma+2}+M^{2}+m^4M^{2\gamma+2}{\left\|(\sqrt{\Phi_0},\psi_0)\right\|}^4
%      +\left\|\frac{\tilde{v}_{0\xi}}{\tilde{v}_0}\right\|^4\right)
% {\left\|(\sqrt{\Phi_0},\psi_0)\right\|}^2\\[2mm]
 \lesssim ~& h(m,M,\phi_0,\psi_0)+(m^{\gamma-1}+M^{\gamma-1})
\int_0^t\int_0^{\infty}\frac{\tilde{v}_{\xi}^2}{v^{\gamma}\tilde{v}^2}\mathrm{d}\xi \mathrm{d}\tau,
 \end{aligned}
\end{equation}
with
\begin{equation}\label{h}
  h(m,M,\phi_0,\psi_0)=\|\psi_{0\xi}\|^2+\left[m^{2\gamma+2}+M^{2}+m^4M^{2\gamma+2}{\left\|(\sqrt{\Phi_0},\psi_0)\right\|}^4
      +\left\|\frac{\tilde{v}_{0\xi}}{\tilde{v}_0}\right\|^4\right]
 {\left\|(\sqrt{\Phi_0},\psi_0)\right\|}^2.
\end{equation}
Plugging (\ref{3rdEBL+}) into (\ref{2ndEstimate'BL+}) gives the following lemma.
 \begin{Lemma}
  There exists $c_0>0$ independent of $T,\ m,\ M$ and $a$,
  such that if
  \begin{equation} \label{conditionBL+}
    g_0(a,m,M,\phi_0,\psi_0):=M^2{\left\|(\sqrt{\Phi_0},\psi_0)\right\|}^4
    +M^2a^{-1}(m^{\gamma-1}+M^{\gamma-1})\leq c_0,
  \end{equation}
  it holds that for each $a>1$ and $0\leq t\leq T$,
  \begin{equation} \label{2ndEstimate''BL+}
%    \begin{aligned}
      \left\|\frac{\tilde{v}_{\xi}}{\tilde{v}}(t)\right\|^2
      +\int_0^t\int_0^{\infty}\frac{\tilde{v}_{\xi}^2}{v^{\gamma}\tilde{v}^2}\mathrm{d}\xi \mathrm{d}\tau
      \lesssim \left\|\frac{\tilde{v}_{0\xi}}{\tilde{v}_0}\right\|^2
      +a{\left\|(\sqrt{\Phi_0},\psi_0)\right\|}^2+M^2a^{-1} h(m,M,\phi_0,\psi_0)
%    \end{aligned}
  \end{equation}
  and
  \begin{equation} \label{3rdEstimate'BL+}
    \begin{aligned}
      &\|\psi_{\xi}(t)\|^2+\int_0^t|s_-|\psi_{\xi}^2(\tau,0)\mathrm{d}\tau
      +\int_0^t\left\|\frac{\psi_{\xi\xi}}{\sqrt{v}}(\tau)\right\|^2\mathrm{d}\tau\\[2mm]
      \lesssim~ &
      h(m,M,\phi_0,\psi_0)+(m^{\gamma-1}+M^{\gamma-1})
      \left[\left\|\frac{\tilde{v}_{0\xi}}{\tilde{v}_0}\right\|^2
      +a{\left\|(\sqrt{\Phi_0},\psi_0)\right\|}^2\right],
    \end{aligned}
  \end{equation}
  where $h(m,M,\phi_0,\psi_0)$ is given by (\ref{h}).
\end{Lemma}

\noindent
\textbf{Proof of Theorem 1.}
Without loss of generality, we assume that $\epsilon\leq 1$.
First, we note from (\ref{Phiphi}), (\ref{log(tildeofv)}) and the initial conditions (\ref{ic1}) that
$${\left\|(\sqrt{\Phi_0},\psi_0)\right\|}\leq C\epsilon^{\alpha-\frac{\gamma+1}{2}l},\quad
\left\|\frac{\tilde{v}_{0\xi}}{\tilde{v}_0}\right\|\leq C\epsilon^{-l}(\|\phi_{0\xi}\|+\|V_{\xi}\phi\|)
\leq C\epsilon^{-l-\beta}.$$
Since $({\phi}_0,{\psi}_0)\in H_0^1(\mathbb{R}_{+})$,
we apply Lemma 2.2 to find $t_0>0$ such that
the problem (\ref{BLperturbation}) has a unique solution $(\phi,\psi)\in X_{1/m_0,M_0}(0,t_0)$
with $m_0,M_0\lesssim\epsilon^{-l}$.
Then we find that for $a=\epsilon^{-2\alpha-2\beta+(\gamma-1)l}>1$,
$$g_0\left(a,m_0,M_0,{\phi}_0,{\psi}_0\right)\lesssim
\epsilon^{4\alpha-2(\gamma+2)l}+\epsilon^{2\alpha+2\beta-2\gamma l}.$$
Hence if $(\ref{indexBL+})_1$ holds, there exists $\epsilon_1>0$ such that
(\ref{conditionBL+}) holds for each $0<\epsilon \leq \epsilon_1$. Furthermore, we have that if
$(\ref{indexBL+})_1$ holds,
$$h(m_0,M_0,\phi_0,\psi_0)\lesssim
\epsilon^{-2\beta}+\left[\epsilon^{-(2\gamma+2)l}+\epsilon^{4\alpha-(4\gamma+8)l}+\epsilon^{-4l-4\beta}\right]
    \epsilon^{2\alpha-(\gamma+1)l}\lesssim \epsilon^{-2\alpha-4\beta+(\gamma-1)l}.$$
Clearly we conclude from Lemma 3.3 that
for each $0<\epsilon \leq \epsilon_1$,
    $$
       \left\|\frac{\tilde{v}_{\xi}}{\tilde{v}}(t)\right\|^2
      +\int_0^t\int_0^{\infty}\frac{\tilde{v}_{\xi}^2}{v^{\gamma}\tilde{v}^2}\mathrm{d}\xi \mathrm{d}\tau
      \lesssim \epsilon^{-2l-2\beta}.
    $$
To get the upper and lower bounds for the specific volume $v$, we follow \cite{M-N2000} and introduce
\begin{equation}\label{Psi}
  \Psi(\tilde{v}):=\int_1^{\tilde{v}}\frac{\sqrt{\tilde{\Phi}(\eta)}}{\eta}\mathrm{d}\eta,
\end{equation}
where $\tilde{\Phi}$ is defined by (\ref{tildeofPhi}).
Then we have from $\tilde{v}(t_0,\infty)=1$ and Lemma 3.1 that
\begin{equation} \label{boundofPsi}
  \begin{aligned}
    \left|\Psi(\tilde{v}(t_0,\xi))\right|
    =\left|\int_{\infty}^{\xi}\Psi_{\xi}(\tilde{v}(t_0,\xi))\mathrm{d}\xi\right|
        \leq \left\| \sqrt{\tilde{\Phi}}(t_0) \right\| \left\| \frac{\tilde{v}_{\xi}}{\tilde{v}}(t_0) \right\|
    \lesssim \epsilon^{\theta},
  \end{aligned}
\end{equation}
where $\theta=\alpha-\beta-(\gamma+3)l/2.$
Since $\Psi(\tilde{v})=O(\tilde{v}^{1/2})$ as $\tilde{v}\to \infty$ and
$\Psi(\tilde{v})=O(\tilde{v}^{(1-\gamma)/2})$ as $\tilde{v}\to 0+$,
%\begin{equation}
%  \tilde{\Phi}(\eta)\sim
%  \begin{cases}
%    \eta,\qquad\quad\ \! \eta\to\infty,\\
%    \eta^{-\gamma+1},\quad \eta\to 0+.
%  \end{cases}
%\end{equation}
(\ref{boundofPsi}) yields
\begin{equation} \label{boundofvBL+}
  \epsilon^{{2\theta}/(1-\gamma)}\lesssim v (t_0,\xi)\lesssim\epsilon^{2\theta},\quad\forall\ \xi\in\mathbb{R}_+.
\end{equation}

Since $(\ref{indexBL+})_2$ implies $2\theta/(\gamma-1)\leq 0$ and $ 2\theta \leq 0$,
we exploit Lemma 2.2 again and recall (\ref{boundofvBL+}) to find $t_1>0$
such that (\ref{BLperturbation}) has a unique solution $(\phi,\psi)\in X_{1/m_1,M_1}(0,t_0+t_1)$,
where $m_1 \lesssim \epsilon^{2\theta/(\gamma-1)}$ and $ M_1 \lesssim \epsilon ^{2\theta}$.
Note that $\theta\leq 0$ implies that
${\epsilon^{4\theta}}\geq\epsilon^{\frac{4\gamma+4}{\gamma-1}\theta}$.
Thus we have
$$h(m_1,M_1,\phi(t_0),\psi(t_0))\lesssim \epsilon^{-2\beta}+\left[\epsilon^{\frac{4\gamma+4}{\gamma-1}\theta}
    +\epsilon^{\frac{8\theta}{\gamma-1}+(4\gamma+4)\theta+4\alpha-2(\gamma+1)l}+\epsilon^{-4l-4\beta}\right]
    \epsilon^{2\alpha-(\gamma+1)l},$$
and therefore the right-hand side of (\ref{2ndEstimate''BL+}) is bounded by $C\epsilon^{-2l-2\beta}$.
 By elementary calculations, we conclude from $(\ref{indexBL+})_{1,3,4}$ that
   $$h(m_1,M_1,\phi(t_0),\psi(t_0))\lesssim \epsilon^{-4\theta-2\alpha-4\beta+(\gamma-3)l}.$$
We recall $a=\epsilon^{-2\alpha-2\beta+(\gamma-1)l}>1$ to deduce
$$g_0(a,m_1,M_1,\phi(t_0),\psi(t_0))
\lesssim \epsilon^{4\theta+4\alpha-2(\gamma+1)l}+\epsilon^{6\theta+2\alpha+2\beta-(\gamma-1) l}
+\epsilon^{(2\gamma+2)\theta+2\alpha+2\beta-(\gamma-1) l}.$$
Then we have from $(\ref{indexBL+})_4$ that there exists $\epsilon_0>0$ such that
(\ref{conditionBL+}) holds for each $\epsilon \leq \epsilon_0$.
Combining Lemma 2.2 and the continuation process,
we can prove (\ref{BL}) has the global-in-time solution
$(\phi,\psi)\in X_{1/m_1,M_1}(0,\infty)$
satisfying
\begin{equation}
  \label{estimate}
  \left\|(\phi,\psi)(t)\right\|^2+\int_0^t\|(|U_{\xi}|^{1/2}\phi,\phi_{\xi},\psi,\psi_{\xi})(\tau)\|^2\mathrm{d}\tau
  \leq C,\quad \forall\ t\in\mathbb{R}_+,
\end{equation}
the constant $C$ depending only on $\epsilon$.
And so the asymptotic behavior of the solution (\ref{largetime})
is concluded by employing Sobolev's inequality.$\quad\quad \Box$

\subsection{Proof of Theorem 2}
In this subsection,
we assume that $(v_-,u_-)\in \Omega_{sub}$, $(v_+,u_+)\in BL_{-}(v_-,u_-)$ and
(\ref{BLperturbation}) has a solution
$(\phi,\psi)\in X_{1/m,M}(0,T)$ satisfying (\ref{Equality}) for some $T>0$ and each $0\leq t\leq T$.
As in the proof of Theorem 1, $c$ and $C$ are used to denote some positive constants
independent of $T,\ m,\ M$ and $\delta$ and the notation $A\lesssim B$ stands for that
$A\leq CB$ holds uniformly for some positive constant independent of $T,\ m,\ M$ and $\delta$.
Without loss of generality, we choose $m$ and $M$ such that $m,M\geq 1$. %%%  $1/m\leq v_+<v_-\leq M$.

The basic energy estimate is stated as follows.
\begin{Lemma}
If $\delta$ is suitably small,
then it holds that for each $0\leq t\leq T$,
\begin{equation} \label{1stEstimateBL-}
\begin{aligned}
  &\left\|(\sqrt{\Phi},\psi)(t)\right\|^2
  +\int^t_0\int_0^{\infty}\left[\frac{{\psi}^2_{\xi}}{v}+
  \frac{|U_{\xi}\phi{\psi}_{\xi}|}{v}+|U_{\xi}|(p(v)-p(V)-p'(V)\phi)\right]\mathrm{d}\xi \mathrm{d}\tau\\[2mm]
  \lesssim ~&{\left\|(\sqrt{\Phi_0},\psi_0)\right\|}^2+\delta m^{\gamma+2}\int_0^{t}\|\phi_{\xi}(\tau)\|^2\mathrm{d}\tau,
\end{aligned}
\end{equation}
where $\Phi=\Phi(v,V)$ is defined by (\ref{Phi}) and $\delta=|u_+-u_-|$ is the strength of the boundary layer
$(V,U)$.
\end{Lemma}
\begin{proof}
First we recall
$(v_+,u_+)\in BL_{-}(v_-,u_-)$
and thus we have $U_{\xi}<0$.
Noting that (\ref{fvV}) and (\ref{BLdecay}) hold,
we conclude that
if $\delta$ is suitably small,
the discriminant $D$ of
$$\mu\frac{\psi_{\xi}^2}{v}-\mu\frac{U_{\xi}\phi\psi_{\xi}}{vV}+|U_{\xi}|(p(v)-p(V)-p'(V)\phi)$$
satisfies
$$D=\frac{\mu |U_{\xi}|}{V^2vf(v,V)}-4\leq \frac{C\delta}{\gamma V^{-\gamma+1}}-4<0,$$
where $f(v,V)$ is defined by (\ref{f}).
Then we integrate (\ref{BLdivergence1}) over $(0,t)\times(0,\infty)$ to find that
\begin{equation} \label{1stEstBL-}
\begin{aligned}
  &\left\|(\sqrt{\Phi},\psi)(t)\right\|^2
  +\int^t_0\int_0^{\infty}\left[\frac{{\psi}^2_{\xi}}{v}+
  \frac{|U_{\xi}\phi{\psi}_{\xi}|}{v}+|U_{\xi}|(p(v)-p(V)-p'(V)\phi)\right]\mathrm{d}\xi \mathrm{d}\tau\\[2mm]
  \lesssim ~&{\left\|(\sqrt{\Phi_0},\psi_0)\right\|}^2+
  \int_0^{t}\int_0^{\infty}|U_{\xi}|(p(v)-p(V)-p'(V)\phi)\mathrm{d}\xi \mathrm{d}\tau.
\end{aligned}
\end{equation}
To estimate the last term of (\ref{1stEstBL-}),
we apply  the idea in Nikkuni and Kawashima~\cite{Nikkuni-Kawashima}, i.e.,
\begin{equation} \label{kawashimaTech}
\phi(t,\xi)=\phi(t,0)+\int^\xi_0{\phi_\xi(t,\eta)}\mathrm{d}\eta\leq \xi^{\frac{1}{2}}\|\phi_\xi(t)\|.
\end{equation}
Since
\begin{equation}
  p(v)-p(V)-p'(V)\phi=\int_{0}^1\int_{0}^1p''(\theta_1\theta_2v+(1-\theta_1\theta_2)V)\theta_1\mathrm{d}\theta_2
  \mathrm{d}\theta_1\phi^2\lesssim m^{\gamma+2}\phi^2,
\end{equation}
we deduce from $|U_{\xi}(\xi)|\lesssim \delta e^{-c\xi}$ and (\ref{kawashimaTech}) that
\begin{equation} \label{3.2.1}
  \begin{aligned}
  \int_0^{t}\int_0^{\infty}|U_{\xi}|(p(v)-p(V)-p'(V)\phi)\mathrm{d}\xi \mathrm{d}\tau
  \lesssim  \delta m^{\gamma+2}\int_0^{t}\int_0^{\infty}  e^{-c\xi}\phi^2 \mathrm{d}\xi \mathrm{d}\tau
    \lesssim \delta m^{\gamma+2}\int_0^{t}\|\phi_{\xi}(\tau)\|^2\mathrm{d}\tau.
  \end{aligned}
\end{equation}
We then substitute (\ref{3.2.1}) into (\ref{1stEstBL-}) to find (\ref{1stEstimateBL-}).
\end{proof}
We now estimate the last term of (\ref{1stEstimateBL-}). For this,
we apply $|V_{\xi}(\xi)|\lesssim \delta e^{-c\xi}$ and (\ref{kawashimaTech}) to conclude
\begin{equation} \label{3.2.8}
   \int_0^t\int_0^{\infty}V_{\xi}^2\phi^2\mathrm{d}\xi \mathrm{d}\tau
  \lesssim\delta^2\int_0^t\|\phi_{\xi}(\tau)\|^2\mathrm{d}\tau.
\end{equation}
Then, we obtain from (\ref{log(tildeofv)}) that
\begin{equation}
  \int_0^t\|\phi_{\xi}(\tau)\|^2\mathrm{d}\tau
  \lesssim M^{\gamma+2}\int_0^t\int_0^{\infty}\frac{\tilde{v}_{\xi}^2}{v^{\gamma}\tilde{v}^2}\mathrm{d}\xi \mathrm{d}\tau
  +\delta^2\int_0^t\|\phi_{\xi}(\tau)\|^2\mathrm{d}\tau.
\end{equation}
Consequently if $\delta$ is sufficiently small,
\begin{equation} \label{3.2.3}
\begin{aligned}
  \int_0^t\|\phi_{\xi}(\tau)\|^2\mathrm{d}\tau
  \lesssim M^{\gamma+2}\int_0^t\int_0^{\infty}\frac{\tilde{v}_{\xi}^2}{v^{\gamma}\tilde{v}^2}\mathrm{d}\xi \mathrm{d}\tau.
\end{aligned}
\end{equation}
Plugging (\ref{3.2.3}) into (\ref{1stEstimateBL-}) gives that if $\delta$ is sufficiently small, it holds that
\begin{equation} \label{1BL-}
\begin{aligned}
  &\left\|(\sqrt{\Phi},\psi)(t)\right\|^2
  +\int^t_0\int_0^{\infty}\left[\frac{{\psi}^2_{\xi}}{v}+
  \frac{|U_{\xi}\phi{\psi}_{\xi}|}{v}+|U_{\xi}|(p(v)-p(V)-p'(V)\phi)\right]\mathrm{d}\xi \mathrm{d}\tau\\[2mm]
  \lesssim ~&{\left\|(\sqrt{\Phi_0},\psi_0)\right\|}^2+\delta m^{\gamma+2}M^{\gamma+2}\int_0^t\int_0^{\infty}\frac{\tilde{v}_{\xi}^2}{v^{\gamma}\tilde{v}^2}\mathrm{d}\xi \mathrm{d}\tau.
\end{aligned}
\end{equation}
Similar to proving Lemma 3.2, we obtain from (\ref{1BL-}) that
if $\delta$ is suitably small,
\begin{equation*} \label{2ndEstimate1BL-}
  \begin{aligned}
    \left\|\frac{\tilde{v}_{\xi}}{\tilde{v}}(t)\right\|^2+
    \int_0^t\int_0^{\infty}\frac{\tilde{v}_{\xi}^2}{v^{\gamma}\tilde{v}^2}\mathrm{d}\xi \mathrm{d}\tau
    \lesssim \left\|\left(\sqrt{\Phi_0},\psi_0,\frac{\tilde{v}_{0\xi}}{\tilde{v}_0}\right)\right\|^2
    +\delta m^{\gamma+2}M^{\gamma+2}\int_0^t\int_0^{\infty}
    \frac{\tilde{v}_{\xi}^2}{v^{\gamma}\tilde{v}^2}\mathrm{d}\xi \mathrm{d}\tau
    +\int_0^t \psi_{\xi}^2(\tau,0)\mathrm{d}\tau,
  \end{aligned}
\end{equation*}
which implies that if $\delta m^{\gamma+2} M^{\gamma+2}$ is sufficiently small,
\begin{equation} \label{2BL-}
  \begin{aligned}
    \left\|\frac{\tilde{v}_{\xi}}{\tilde{v}}(t)\right\|^2+
    \int_0^t\int_0^{\infty}\frac{\tilde{v}_{\xi}^2}{v^{\gamma}\tilde{v}^2}\mathrm{d}\xi \mathrm{d}\tau
    \lesssim \left\|\left(\sqrt{\Phi_0},\psi_0,\frac{\tilde{v}_{0\xi}}{\tilde{v}_0}\right)\right\|^2
    +\int_0^t \psi_{\xi}^2(\tau,0)\mathrm{d}\tau.
  \end{aligned}
\end{equation}
Now we substitute (\ref{boundaryterm2}) %with $a=\delta^{-2\alpha-2\beta+(\gamma-1)l}$
into (\ref{2BL-})
to conclude the following lemma.
\begin{Lemma}
  There exists a constant $c_1$independent of $T,\ m,\ M$, $\delta$ and $a$,
such that if
\begin{equation} \label{condition1BL-}
  g_1(\delta,a,m,M):=a\delta m^{\gamma+2} M^{\gamma+2}\leq c_1,
\end{equation}
then it holds for each $a>1$ and $0\leq t\leq T$,
\begin{equation} \label{2ndEstimate3BL-}
  \begin{aligned}
    \left\|\frac{\tilde{v}_{\xi}}{\tilde{v}}(t)\right\|^2+
    \int_0^t\int_0^{\infty}\frac{\tilde{v}_{\xi}^2}{v^{\gamma}\tilde{v}^2}\mathrm{d}\xi \mathrm{d}\tau
    \lesssim\left\|\frac{\tilde{v}_{0\xi}}{\tilde{v}_0}\right\|^2
    +a{\left\|(\sqrt{\Phi_0},\psi_0)\right\|}^2
    +M^2a^{-1}
    \int_0^t\left\|\frac{\psi_{\xi\xi}}{\sqrt{v}}(\tau)\right\|^2\mathrm{d}\tau
  \end{aligned}
  \end{equation}
  and
  \begin{equation} \label{1stEstimate2BL-}
  \begin{aligned}
    &\left\|(\sqrt{\Phi},\psi)(t)\right\|^2
    +\int^t_0\int_0^{\infty}\left[\frac{{\psi}^2_{\xi}}{v}
    +\frac{|U_{\xi}\phi{\psi}_{\xi}|}{v}+|U_{\xi}|(p(v)-p(V)-p'(V)\phi)\right]\mathrm{d}\xi \mathrm{d}\tau\\[2mm]
    \lesssim ~&{\left\|(\sqrt{\Phi_0},\psi_0)\right\|}^2
    +\delta m^{\gamma+2} M^{\gamma+2}\left\|\frac{\tilde{v}_{0\xi}}{\tilde{v}_0}\right\|^2
    +M^2a^{-1}\delta m^{\gamma+2} M^{\gamma+2}
    \int_0^t\left\|\frac{\psi_{\xi\xi}}{\sqrt{v}}(\tau)\right\|^2\mathrm{d}\tau.
  \end{aligned}
  \end{equation}
\end{Lemma}

Next, we will use (\ref{3rdEstimateBL+}) to estimate the term
$\int_0^t\|\frac{\psi_{\xi\xi}}{\sqrt{v}}(\tau)\|^2\mathrm{d}\tau$.
We employ (\ref{3.2.8}) and (\ref{3.2.3}) to obtain
\begin{equation}
  \label{JBL-}
    \begin{aligned}
  \int_0^t\int_0^{\infty}J\phi^2\mathrm{d}\xi \mathrm{d}\tau
  \lesssim\int_0^t\int_0^{\infty}(v^{-2\gamma-3}+v)V_{\xi}^2\phi^2 \mathrm{d}\xi \mathrm{d}\tau%\\[2mm]
    \lesssim(m^{2\gamma+3}+M)\delta^2M^{\gamma+2}\int_0^t
    \int_0^{\infty}\frac{\tilde{v}_{\xi}^2}{v^{\gamma}\tilde{v}^2}\mathrm{d}\xi \mathrm{d}\tau.
  \end{aligned}
\end{equation}
We have from (\ref{1BL-}) that
\begin{equation} \label{3.2.6}
  \begin{aligned}
  \int_0^t\int_0^{\infty}m^2V_{\xi}^2\frac{\psi_{\xi}^2}{v}\mathrm{d}\xi \mathrm{d}\tau
    \lesssim\delta^{2}m^2{\left\|(\sqrt{\Phi_0},\psi_0)\right\|}^2
    +\delta^{3}m^{\gamma+4} M^{\gamma+2}\int_0^t\int_0^{\infty}
    \frac{\tilde{v}_{\xi}^2}{v^{\gamma}\tilde{v}^2}\mathrm{d}\xi \mathrm{d}\tau.
  \end{aligned}
\end{equation}
Substituting (\ref{JBL-}) and (\ref{3.2.6}) into (\ref{3rdEstimateBL+}), we conclude that
if (\ref{condition1BL-}) holds, then
\begin{equation} \label{3.2.9}
  \begin{aligned}
    \|\psi_{\xi}(t)\|^2+\int_0^t\left\|\frac{\psi_{\xi\xi}}{\sqrt{v}}(\tau)\right\|^2\mathrm{d}\tau
    \lesssim ~&\left\|(\sqrt{\Phi_0},\psi_0,\psi_{0\xi})\right\|^2\\
    &+m^{\gamma-1}\int_0^t\int_0^{\infty}\frac{\tilde{v}_{\xi}^2}{v^{\gamma}\tilde{v}^2}\mathrm{d}\xi \mathrm{d}\tau
    &+\int_0^t\left\|\frac{\phi_{\xi}}{v}\right\|^4\left\|\frac{\psi_{\xi}}{\sqrt{v}}\right\|^2\mathrm{d}\tau.
  \end{aligned}
\end{equation}
Then combinations of (\ref{2ndEstimate3BL-}) and (\ref{3.2.9})
give the following lemma.
\begin{Lemma}
  Assume that (\ref{condition1BL-}) hold.
  There exists a constant $c_2$ independent of $T,\ m,\ M$, $\delta$ and $a$, such that if
  \begin{equation} \label{condition2BL-}\begin{aligned}
    g_2(a,m,M):=
    a^{-1}m^{\gamma-1}M^2\leq c_2,
  \end{aligned}
  \end{equation}
  then it holds for each $a>1$ and $0\leq t\leq T$,
  \begin{equation} \label{3rdEstimateBL-}
  \begin{aligned}
    \|\psi_{\xi}(t)\|^2+\int_0^t\left\|\frac{\psi_{\xi\xi}}{\sqrt{v}}(\tau)\right\|^2\mathrm{d}\tau
    \lesssim ~&\|\psi_{0\xi}\|^2
    +m^{\gamma-1}a{\left\|(\sqrt{\Phi_0},\psi_0)\right\|}^2\\[2mm]
    &+m^{\gamma-1}\left\|\frac{\tilde{v}_{0\xi}}{\tilde{v}_0}\right\|^2
    +\int_0^t\left\|\frac{\phi_{\xi}}{v}\right\|^4\left\|\frac{\psi_{\xi}}{\sqrt{v}}\right\|^2\mathrm{d}\tau.
  \end{aligned}
\end{equation}
\end{Lemma}
We now need to estimate the last term in (\ref{3rdEstimateBL-}).
First substitute (\ref{2ndEstimate3BL-}) and (\ref{1stEstimate2BL-}) into (\ref{3.1.2}), and then recall
%(\ref{condition1BL-}) and
(\ref{condition2BL-}) to deduce
\begin{equation} \label{3.2.4}
  \begin{aligned}
    \left\|\frac{\phi_{\xi}}{v}(\tau)\right\|^2
    \lesssim  \left\|\frac{\tilde{v}_{0\xi}}{\tilde{v}_0}\right\|^2
    +a{\left\|(\sqrt{\Phi_0},\psi_0)\right\|}^2
    &+M^2a^{-1}\int_0^{\tau}\left\|\frac{\psi_{\xi\xi}}{\sqrt{v}}(s)\right\|^2\mathrm{d}s.
  \end{aligned}
\end{equation}
Applying the inequality
\begin{equation}
  (c+b)^3\leq 4(c^3+b^3)\quad \textrm{if}\ c,b\geq 0,
\end{equation}
and noting $\delta m^{\gamma+2}M^{\gamma+2}\lesssim a^{-1}$,
we discover from (\ref{3.2.4}) and (\ref{1stEstimate2BL-}) that
\begin{equation} \label{3.2.00}
  \begin{aligned}
    \int_0^t\left\|\frac{\phi_{\xi}}{v}(\tau)\right\|^4
    \left\|\frac{\psi_{\xi}}{\sqrt{v}}(\tau)\right\|^2\mathrm{d}\tau
    \lesssim ~& a^{-1}\left[
    \left\|\frac{\tilde{v}_{0\xi}}{\tilde{v}_0}\right\|^2
    +a{\left\|(\sqrt{\Phi_0},\psi_0)\right\|}^2\right]^3\\[2mm]
    &+a^{-1}\left[M^2a^{-1}\int_0^t\left\|\frac{\psi_{\xi\xi}}{\sqrt{v}}(\tau)\right\|^2
    \mathrm{d}\tau\right]^3.
  \end{aligned}
\end{equation}
To close the energy estimate (\ref{3rdEstimateBL-}),
we will employ the following result due to Strauss~\cite{Strauss}.
\begin{Lemma}
\label{Strauss}
%(\cite{Strauss})
Let $M(t)$ be an non-negative continuous function of $t$ satisfying the
inequality
\begin{equation}
  M(t)\lesssim A_1 + A_2 M(t)^{\kappa}
\end{equation}
in some interval containing $0$, where $A_1$ and $A_2$ are positive
constants and $\kappa>1$.
Then there is a positive constant $C$ such that if $ M(0)\lesssim A_1$ and
\begin{equation}
    A_1 A_2^{\frac{1}{\kappa-1}}< C (1-\kappa^{-1})\kappa^{-\frac{1}{\kappa-1}},
\end{equation}
then in the same interval
\begin{equation}
  M(t)\lesssim \frac{A_1}{1-\kappa^{-1}}.
\end{equation}
\end{Lemma}
Now we write
\begin{equation}
  M(t)=\|\psi_{\xi}(t)\|^2+\int_0^t\left\|\frac{\psi_{\xi\xi}}{\sqrt{v}}(\tau)\right\|^2\mathrm{d}\tau,
\end{equation}
and
\begin{equation} \label{h1}
  \begin{aligned}
    h_1
    =\|\psi_{0\xi}\|^2
    +m^{\gamma-1}a{\left\|(\sqrt{\Phi_0},\psi_0)\right\|}^2
    +m^{\gamma-1}\left\|\frac{\tilde{v}_{0\xi}}{\tilde{v}_0}\right\|^2
    +a^{-1}\left[\left\|\frac{\tilde{v}_{0\xi}}{\tilde{v}_0}\right\|^2
    +a{\left\|(\sqrt{\Phi_0},\psi_0)\right\|}^2\right]^3.
  \end{aligned}
\end{equation}
%and
%\begin{equation} \label{A2}
%  A_2=a^{-4}M^6.
%\end{equation}
Then we substitute (\ref{3.2.00}) into (\ref{3rdEstimateBL-}) to find
\begin{equation}
  M(t)\lesssim h_1 + a^{-4}M^6 M(t)^3.
\end{equation}
Noting that $M(0)\lesssim h_1$ holds,
we can close
the estimates (\ref{3rdEstimateBL-}), (\ref{2ndEstimate3BL-}) and (\ref{1stEstimate2BL-})
by applying Lemma \ref{Strauss}.
\begin{Lemma}
\label{LemmaBL-}
  Suppose that (\ref{condition1BL-}) and (\ref{condition2BL-}) hold.
  There is a  constant $c_3$ independent of $T$, $m$, $M$, $\delta$ and $a$, such that
  if
  \begin{equation} \label{condition3BL-}
    g_3(\delta,a, m, M,\phi_0,\psi_0):=h_1 {a^{-2}M^3} < c_3,
  \end{equation}
  where $A_1$ is defined by (\ref{h1}),
  then it holds that for each $a>1$ and $ 0\leq t\leq T$,
  \begin{equation}
    \|\psi_{\xi}(t)\|^2+\int_0^t\left\|
    \frac{\psi_{\xi\xi}}{\sqrt{v}}(\tau)\right\|^2\mathrm{d}\tau\lesssim h_1,
  \end{equation}
  \begin{equation} \label{2ndEstimate4BL-}
  \begin{aligned}
    \left\|\frac{\tilde{v}_{\xi}}{\tilde{v}}(t)\right\|^2+
    \int_0^t\int_0^{\infty}\frac{\tilde{v}_{\xi}^2}{v^{\gamma}\tilde{v}^2}\mathrm{d}\xi \mathrm{d}\tau
    \lesssim\left\|\frac{\tilde{v}_{0\xi}}{\tilde{v}_0}\right\|^2
    +a{\left\|(\sqrt{\Phi_0},\psi_0)\right\|}^2
    +M^2a^{-1}
    h_1
  \end{aligned}
  \end{equation}
  and
  \begin{equation} \label{1stEstimate3BL-}
  \begin{aligned}
    &\left\|(\sqrt{\Phi},\psi)(t)\right\|^2
    +\int^t_0\int_0^{\infty}\left[\frac{{\psi}^2_{\xi}}{v}
    +\frac{|U_{\xi}\phi{\psi}_{\xi}|}{v}+|U_{\xi}|(p(v)-p(V)-p'(V)\phi)\right]\mathrm{d}\xi \mathrm{d}\tau\\[2mm]
    \lesssim ~&{\left\|(\sqrt{\Phi_0},\psi_0)\right\|}^2
    +\delta m^{\gamma+2} M^{\gamma+2}\left\|\frac{\tilde{v}_{0\xi}}{\tilde{v}_0}\right\|^2
    +M^2a^{-1}\delta m^{\gamma+2} M^{\gamma+2}
    h_1.
  \end{aligned}
  \end{equation}
\end{Lemma}

\noindent
\textbf{Proof of Theorem 2.}
Without loss of generality, we assume that $\delta\leq 1$.
First, we note from (\ref{Phiphi}), (\ref{log(tildeofv)}) and the initial conditions (\ref{ic2}) that
$${\left\|(\sqrt{\Phi_0},\psi_0)\right\|}\leq C\delta^{\alpha-\frac{\gamma+1}{2}l},\quad
\left\|\frac{\tilde{v}_{0\xi}}{\tilde{v}_0}\right\|\leq C\delta^{-l}(\|\phi_{0\xi}\|+\|V_{\xi}\phi\|)
\leq C\delta^{-l-\beta}.$$
Since $({\phi}_0,{\psi}_0)\in H_0^1(\mathbb{R}_{+})$,
we apply Lemma 2.2 to find $t_0>0$ such that
the problem (\ref{BLperturbation}) has a unique solution $(\phi,\psi)\in X_{1/m_0,M_0}(0,t_0)$
with $m_0,M_0\lesssim\delta^{-l}$.
Then we find that for $a=\delta^{-2\alpha-2\beta+(\gamma-1)l}>1$,
$$g_1\left(\delta,a,m_0,M_0\right)
+g_2(a,m_0,M_0)\lesssim
\delta^{1-2\alpha-2\beta-(\gamma+5)l}+\delta^{2\alpha+2\beta-2\gamma l},$$
and
$$g_3(\delta,a, m_0, M_0,\phi_0,\psi_0)\lesssim
\left[\delta^{-2\beta-(\gamma+1)l}+\delta^{2\alpha-4\beta-(\gamma+5)l}\right]
 \delta^{4\alpha+4\beta-(2\gamma+1)l}.$$
Hence if $(\ref{indexBL-})_1$ holds, there exists $\delta_1>0$ such that
(\ref{condition1BL-}), (\ref{condition2BL-}) and (\ref{condition3BL-})
hold for each $\delta \leq \delta_1$.
Next we compute from %$  \| {\tilde{v}_{0\xi}}/{\tilde{v}_0} \|  \leq C\epsilon^{-l-\beta}$ and
$(\ref{indexBL-})_{1}$ that the right-hand sides of (\ref{2ndEstimate4BL-}) and (\ref{1stEstimate3BL-})
are bounded by $C_1\delta^{-2\beta-2l}$ and $C_2\delta^{2\alpha-(\gamma+1)l}$, respectively.
We conclude from Lemma \ref{LemmaBL-} that
\begin{equation}
  \begin{aligned}
    \left|\Psi(\tilde{v}(t_0,\xi))\right|
    \leq \left\| \sqrt{\tilde{\Phi}}(t_0) \right\| \left\| \frac{\tilde{v}_{\xi}}{\tilde{v}}(t_0) \right\|
    \lesssim\delta^{\theta},
  \end{aligned}
\end{equation}
where $\Psi$ is defined by (\ref{Psi}) and $\theta=\alpha-\beta-(\gamma+3)l/2.$
Hence
\begin{equation} \label{boundofvBL-}
  \delta^{{2\theta}/(1-\gamma)}\lesssim v (t_0,\xi)\lesssim \delta^{2\theta},
  \quad \forall\ \xi\in\mathbb{R}_+.
\end{equation}

Since $(\ref{indexBL-})_2$ implies $2\theta/(\gamma-1)\leq 0$ and $ 2\theta \leq 0$,
we apply Lemma 2.2 again and recall (\ref{boundofvBL+}) to find $t_1>0$
such that (\ref{BLperturbation}) has a unique solution $(\phi,\psi)\in X_{1/m_1,M_1}(0,t_0+t_1)$,
where $m_1 \lesssim \delta^{2\theta/(\gamma-1)}$ and $ M_1 \lesssim \delta ^{2\theta}$.
By elementary calculations, we conclude from $(\ref{indexBL-})_3$ that there exists
$0<\delta_2\leq \delta_1$ such that
that if $\delta \leq \delta_2$, then
the right-hand side of (\ref{2ndEstimate4BL-})  and (\ref{1stEstimate3BL-})
are bounded by $C_1\delta^{-2\beta-2l}$ and $C_2\delta^{2\alpha-(\gamma+1)l}$, respectively.
Since
\begin{equation*}
  g_1\left(\delta,a,m_1,M_!\right)
+g_2(a,m_1,M_1)\lesssim\delta^{1-2\alpha-2\beta+(\gamma-1)l+\frac{2\gamma (\gamma+2)}{\gamma-1}\theta }
  +\delta^{2\alpha+2\beta-(\gamma-1)l+6\theta}
\end{equation*}
and
\begin{equation*}
  g_3(\delta,a, m_1, M_1,\phi(t_0),\psi(t_0))\lesssim
  \left[\delta^{2\theta-2\beta-2l}+\delta^{2\alpha-4\beta-(\gamma+5)l}\right]
  \delta^{4\alpha+4\beta-(2\gamma-2)l+6\theta},
\end{equation*}
$(\ref{indexBL-})_3$ implies that there exists $0<\delta_0\leq \delta_2$ such that
(\ref{condition2BL-}) and (\ref{condition3BL-}) hold for each $\delta \leq \delta_0$.
Then, combining Lemma 2.2 and the continuation process,
we can prove (\ref{BL}) has the global solution
in time $(\phi,\psi)\in X_{1/m_1,M_1}(0,\infty)$
satisfying
(\ref{estimate}) with the constant $C$ depending only on $\delta$.
Thus, the asymptotic behavior of the solution (\ref{largetime})
is concluded by employing Sobolev's inequality.$\quad\quad \Box$

\section{Stability of Rarefaction Wave}

In this section, we investigate the case when
$(v_-,u_-)\in \Omega_{super}$ and $(v_+,u_+)\in R_1(v_-,u_-)$.
We assume that
(\ref{RWperturbation}) has a solution
$(\phi,\psi)\in X_{1/m,M}(0,T)$ satisfying (\ref{Equality}) for some $T>0$ and each $0\leq t\leq T$.
As before, we also simply write $c$ and $C$ as positive constants
independent of $T,\ m,\ M$ and $\epsilon$ and the notation $A\lesssim B$ will mean that
$A\leq CB$ holds uniformly for some positive constant independent of $T,\ m,\ M$ and $\epsilon$.
Without loss of generality, we choose $m$ and $M$ such that $m,M\geq 1$. %%%  $1/m\leq v_-<v_+\leq M$.

Our first result is concerned with the basic energy estimates which is stated as folloes
\begin{Lemma}
  If ${\epsilon}$ is suitably small,
  then it holds for each $0\leq t\leq T$,
%  \begin{equation} %\label{1stEstimateRW}
%\begin{aligned}
%  \left\|(\sqrt{\Phi},\psi)(t)\right\|^2
%  +\int^t_0\int_0^{\infty}\left[\frac{{\psi}^2_{\xi}}{v}+\frac{|U_{\xi}\phi{\psi}_{\xi}|}{v}
%  +|U_{\xi}|(p(v)-p(V)-p'(V)\phi)\right]\mathrm{d}\xi \mathrm{d}\tau
%  \lesssim {\left\|(\sqrt{\Phi_0},\psi_0)\right\|}^2+{\epsilon}^{\frac{2}{9}}M^{\frac{1}{2}}.
%\end{aligned}
%\end{equation}
  \begin{multline}\label{1stEstimateRW}
  \left\|(\sqrt{\Phi},\psi)(t)\right\|^2
  +\int^t_0\int_0^{\infty}\left[\frac{{\psi}^2_{\xi}}{v}+\frac{|U_{\xi}\phi{\psi}_{\xi}|}{v}
  +|U_{\xi}|(p(v)-p(V)-p'(V)\phi)\right]\mathrm{d}\xi \mathrm{d}\tau\\
  \lesssim {\left\|(\sqrt{\Phi_0},\psi_0)\right\|}^2+{\epsilon}^{\frac{2}{9}}M^{\frac{1}{2}}.
\end{multline}
\end{Lemma}
\begin{proof}
  Multiply $(\ref{RWperturbation})_1$ and $(\ref{RWperturbation})_2$ by $p(V)-p(v)$ and $\psi$, respectively,
  and add these two equations to have
  \begin{equation} \label{divergence1RW}
  \begin{aligned}
  &\left[\Phi+\frac{1}{2}\psi^2\right]_t+\mu\frac{\psi_{\xi}^2}{v}-\mu\frac{U_{\xi}\phi\psi_{\xi}}{vV}
  +U_{\xi}(p(v)-p(V)-p'(V)\phi)\\
  =&\left[s_-\left(\Phi+\frac{1}{2}\psi^2\right)+(p(V)-p(v))\psi
  +\mu\left(\frac{U_{\xi}+{\psi}_{\xi}}{V+\phi}-\frac{U_{\xi}}{V}\right)\psi\right]_{\xi}
  +\mu\psi\left(\frac{U_{\xi}}{V}\right)_{\xi}.
  \end{aligned}
  \end{equation}
  We recall $U_{\xi}\geq 0$ and thus compute from (\ref{fvV}) and (\ref{Uxip}) that
  if ${\epsilon}$ is suitably small, the discriminant $D$ of
$$\mu\frac{\psi_{\xi}^2}{v}-\mu\frac{U_{\xi}\phi\psi_{\xi}}{vV}+U_{\xi}(p(v)-p(V)-p'(V)\phi)$$
satisfies
  $$D=\frac{\mu |U_{\xi}|}{V^2vf(v,V)}-4\leq \frac{C{\epsilon}}{V^{-\gamma+1}}-4<0.$$
  Then, we integrate (\ref{divergence1RW}) over $(0,t)\times(0,\infty)$ to find
%    \begin{equation} \label{1stEstRW}
%\begin{aligned}
%  &\left\|(\sqrt{\Phi},\psi)(t)\right\|^2
%  +\int^t_0\int_0^{\infty}\left[\frac{{\psi}^2_{\xi}}{v}+\frac{|U_{\xi}\phi{\psi}_{\xi}|}{v}
%  +|U_{\xi}|(p(v)-p(V)-p'(V)\phi)\right]\mathrm{d}\xi \mathrm{d}\tau\\[2mm]
%  \lesssim ~&{\left\|(\sqrt{\Phi_0},\psi_0)\right\|}^2+
%  \int_0^t\|\psi(\tau)\|_{L^{\infty}}\left\|\left(\frac{U_{\xi}}{V}\right)_{\xi}(\tau)\right\|_{L^1}
%  \mathrm{d}\tau.
%\end{aligned}
%\end{equation}
    \begin{multline} \label{1stEstRW}
%\begin{aligned}
  \left\|(\sqrt{\Phi},\psi)(t)\right\|^2
  +\int^t_0\int_0^{\infty}\left[\frac{{\psi}^2_{\xi}}{v}+\frac{|U_{\xi}\phi{\psi}_{\xi}|}{v}
  +|U_{\xi}|(p(v)-p(V)-p'(V)\phi)\right]\mathrm{d}\xi \mathrm{d}\tau\\[2mm]
  \lesssim {\left\|(\sqrt{\Phi_0},\psi_0)\right\|}^2+
  \int_0^t\|\psi(\tau)\|_{L^{\infty}}\left\|\left(\frac{U_{\xi}}{V}\right)_{\xi}(\tau)\right\|_{L^1}
  \mathrm{d}\tau.
%\end{aligned}
\end{multline}
  To estimate the last term of (\ref{1stEstRW}),
  we get from (\ref{Uxixip}) and $q\geq 10$ that
  \begin{equation*}
    \left\|\left(\frac{U_{\xi}}{V}\right)_{\xi}(\tau)\right\|_{L^1}
    \leq \left\|\left(\frac{U_{\xi}}{V}\right)_{\xi}(\tau)\right\|_{L^1}^{\frac{1}{9}+\frac{8}{9}}
    \lesssim {\epsilon}^{\frac{1}{9}}(1+\tau)^{-\frac{4}{5}}.
  \end{equation*}
  Then we employ
  Sobolev's inequality and Young's inequality to deduce that for each $\nu>0$,
  \begin{equation} \label{4.3}
    \begin{aligned}
      &\int_0^t\|\psi(\tau)\|_{L^{\infty}}\left\|\left(\frac{U_{\xi}}{V}\right)_{\xi}(\tau)\right\|_{L^1}
      \mathrm{d}\tau\\[2mm]
      \leq~ &M^{\frac{1}{4}}
      \int_0^t\left\|\frac{\psi_{\xi}}{\sqrt{v}}(\tau)\right\|^{\frac{1}{2}}\|\psi(\tau)\|^{\frac{1}{2}}
      \left\|\left(\frac{U_{\xi}}{V}\right)_{\xi}(\tau)\right\|_{L^1}\mathrm{d}\tau\\[2mm]
      \leq ~&
      \nu\int^t_0\left\|\frac{\psi_{\xi}}{\sqrt{v}}(\tau)\right\|^2\mathrm{d}\tau
      +C(\nu)M^{\frac{1}{3}}\int_0^t \|\psi(\tau)\|^{\frac{2}{3}}
      \left\|\left(\frac{U_{\xi}}{V}\right)_{\xi}(\tau)\right\|_{L^1}^{\frac{4}{3}}\mathrm{d}\tau\\[2mm]
      \leq ~&
      \nu\int^t_0\left\|\frac{\psi_{\xi}}{\sqrt{v}}(\tau)\right\|^2\mathrm{d}\tau
      +\int_0^t \|\psi(\tau)\|^{2}(1+\tau)^{-\frac{16}{15}}\mathrm{d}\tau  %\\&
      +C(\nu){\epsilon}^{\frac{2}{9}}M^{\frac{1}{2}}.
    \end{aligned}
  \end{equation}
  Plugging (\ref{4.3}) into (\ref{1stEstRW}), we can complete this lemma by making use of Gronwall's inequality.
\end{proof}

We set $\tilde{v}:={v}/{V}$, and so
equation $(\ref{RWperturbation})_2$ is also written as
\begin{equation} \label{RW2}
\begin{aligned}
\left[\mu\frac{\tilde{v}_{\xi}}{\tilde{v}}-\psi\right]_t-
s_{-}\left[\mu\frac{\tilde{v}_{\xi}}{\tilde{v}}-\psi\right]_{\xi}
+\frac{\gamma\tilde{v}_{\xi}}{V^{\gamma}{\tilde{v}}^{\gamma+1}}
=\frac{\gamma V_{\xi}}{V^{\gamma+1}}(1-{\tilde{v}}^{-\gamma})-\mu\left(\frac{U_{\xi}}{V}\right)_{\xi}.
\end{aligned}
\end{equation}
Multiplying $(\ref{RW2})$ by ${\tilde{v}_{\xi}}/{\tilde{v}}$, we have a divergence form
\begin{equation} \label{divergence2RW}
  \begin{aligned}
    &\left[\frac{\mu}{2}\left(\frac{\tilde{v}_{\xi}}{\tilde{v}}\right)^2-\psi\frac{\tilde{v}_{\xi}}{\tilde{v}}\right]_t
     +\frac{\gamma \tilde{v}_{\xi}^2}{v^{\gamma}\tilde{v}^2}%\\&
    +\left[\psi\frac{\tilde{v}_t}{\tilde{v}}
     -\frac{\mu s_-}{2}\left(\frac{\tilde{v}_{\xi}}{\tilde{v}}\right)^2\right]_{\xi}\\[2mm]
    =~&\frac{\psi_{\xi}^2}{v}-\frac{U_{\xi}\phi\psi_{\xi}}{vV}
     +\frac{\gamma V_{\xi}}{V^{\gamma+1}}(1-{\tilde{v}}^{-\gamma})\frac{\tilde{v}_{\xi}}{\tilde{v}}
     -\mu\left(\frac{U_{\xi}}{V}\right)_{\xi}\frac{\tilde{v}_{\xi}}{\tilde{v}}.
  \end{aligned}
\end{equation}
We have from (\ref{Phiphi}) and (\ref{Uxip}) that
\begin{equation}
\begin{aligned}
   (1-{\tilde{v}}^{-\gamma})^2=v^{-2\gamma}
   \left[\gamma\int_0^1(\theta v+(1-\theta)V)^{\gamma-1}\mathrm{d}\theta\phi\right]^2
   \lesssim v^{-2\gamma}M^{2\gamma-2}\phi^2\lesssim v^{-2\gamma}M^{3\gamma-1}\Phi,
\end{aligned}
\end{equation}
and
\begin{equation}
  \|V_{\xi}(\tau)\|_{L^{\infty}}\leq \|V_{\xi}(\tau)\|_{L^{\infty}}^{\frac{1}{4}+\frac{3}{4}}
  \lesssim {\epsilon}^{\frac{1}{4}}(1+\tau)^{-\frac{3}{4}}.
\end{equation}
These two inequalities imply that
\begin{equation} \label{4.4}
\begin{aligned}
    &\int_0^t\int_0^{\infty}\frac{\gamma V_{\xi}}{V^{\gamma+1}}(1-{\tilde{v}}^{-\gamma})\frac{\tilde{v}_{\xi}}{\tilde{v}}\mathrm{d}\xi\mathrm{d}\tau\\[2mm]
  \leq ~&
   C m^{\gamma}M^{3\gamma-1} \int_0^t\|V_{\xi}(\tau)\|_{L^{\infty}}^2\|\sqrt{\Phi}(\tau)\|^2\mathrm{d}\tau
   +\int_0^t\int_0^{\infty}\frac{\gamma \tilde{v}_{\xi}^2}{4v^{\gamma}\tilde{v}^2}\mathrm{d}\xi\mathrm{d}\tau\\[2mm]
  \leq ~&
  Cm^{\gamma}M^{3\gamma-1}{\epsilon}^{\frac{1}{2}}
  \left[{\left\|(\sqrt{\Phi_0},\psi_0)\right\|}^2+{\epsilon}^{\frac{2}{9}}M^{\frac{1}{2}}\right]
   +\int_0^t\int_0^{\infty}\frac{\gamma \tilde{v}_{\xi}^2}{4v^{\gamma}\tilde{v}^2}\mathrm{d}\xi\mathrm{d}\tau,
\end{aligned}
\end{equation}
where we used (\ref{1stEstimateRW}) to get the last inequality.
By a similar way as the above, we have from (\ref{Uxixip}) that
\begin{equation} \label{4.5}
  \begin{aligned}
    \int_0^t\int_0^{\infty}\left|\mu\left(\frac{U_{\xi}}{V}\right)_{\xi}
    \frac{\tilde{v}_{\xi}}{\tilde{v}}\right|\mathrm{d}\xi\mathrm{d}\tau%\\[2mm]
    \leq ~&
    CM^{\gamma}\int_0^t\left\|\left(\frac{U_{\xi}}{V}\right)_{\xi}(\tau)\right\|^2\mathrm{d}\tau
    +\int_0^t\int_0^{\infty}\frac{\gamma \tilde{v}_{\xi}^2}{4v^{\gamma}\tilde{v}^2}\mathrm{d}\xi\mathrm{d}\tau\\[2mm]
    \leq ~&
    CM^{\gamma}{\epsilon}
    +\int_0^t\int_0^{\infty}\frac{\gamma \tilde{v}_{\xi}^2}{4v^{\gamma}\tilde{v}^2}\mathrm{d}\xi\mathrm{d}\tau.
  \end{aligned}
\end{equation}
Therefore, integrating (\ref{divergence2RW}) over $(0,t)\times(0,\infty)$ and using
(\ref{4.4})-(\ref{4.5}), we have the following lemma.
\begin{Lemma}
If ${\epsilon}$ is suitably small,
it holds that
\begin{equation} \label{2ndEstimateRW}
\begin{aligned}
    \left\|\frac{\tilde{v}_{\xi}}{\tilde{v}}(t)\right\|^2
    +\int_0^t\int_0^{\infty}\frac{\tilde{v}_{\xi}^2}{v^{\gamma}\tilde{v}^2}\mathrm{d}\xi \mathrm{d}\tau
    \lesssim B+u_-^{-1}\int_0^t \psi_{\xi}^2(\tau,0)\mathrm{d}\tau,
\end{aligned}
\end{equation}
where
\begin{equation} \label{B}
    B=\left\|\frac{\tilde{v}_{0\xi}}{\tilde{v}_0}\right\|^2
    +\left[1+m^{\gamma}M^{3\gamma-1}{\epsilon}^{\frac{1}{2}}\right]
    \left[{\left\|(\sqrt{\Phi_0},\psi_0)\right\|}^2
    +{\epsilon}^{\frac{2}{9}}M^{\frac{1}{2}}\right]+M^{\gamma}\epsilon.
\end{equation}
\end{Lemma}

Multiplying $(\ref{RWperturbation})_2$ by $-\psi_{\xi\xi}$, we get a divergence form
\begin{equation} \label{RWdivergence3}
% \begin{aligned}
   \left[\frac{1}{2}\psi_{\xi}^2\right]_t+\left[\frac{s_-}{2}\psi_{\xi}^2
   -\psi_t\psi_{\xi}\right]_{\xi}+\mu\frac{\psi_{\xi\xi}^2}{v}
   =\sum_{i=1}^3 R_i-\mu\psi_{\xi\xi}\left(\frac{U_{\xi}}{V}\right)_{\xi},
\end{equation}
where $R_i$ $(i=1,2,3)$ are defined in (\ref{BLdivergence3}).
Therefore, integrating (\ref{RWdivergence3}) over $(0,t)\times(0,\infty)$ yields
\begin{equation} \label{3rdEstimateRW}
  \begin{aligned}
    M(t):=~&\|\psi_{\xi}(t)\|^2+u_-\int_0^t\psi_{\xi}^2(\tau,0)\mathrm{d}\tau
    +\int_0^t\left\|\frac{\psi_{\xi\xi}}{\sqrt{v}}(\tau)\right\|^2\mathrm{d}\tau\\[2mm]
    \lesssim~ &\|\psi_{0\xi}\|^2
    +(m^{\gamma-1}+M^{\gamma-1}\epsilon^2)\int_0^t\int_0^{\infty}
    \frac{\tilde{v}_{\xi}^2}{v^{\gamma}\tilde{v}^2}\mathrm{d}\xi \mathrm{d}\tau+M\epsilon  \\[2mm]
    &+\int_0^t\int_0^{\infty}\left[J\phi^2+m^2V_{\xi}^2\frac{\psi_{\xi}^2}{v}\right]\mathrm{d}\xi \mathrm{d}\tau
    +\int_0^t\left\|\frac{\phi_{\xi}}{v}\right\|^4\left\|\frac{\psi_{\xi}}{\sqrt{v}}\right\|^2\mathrm{d}\tau,
  \end{aligned}
\end{equation}
where $J$ is defined by (\ref{J}).
Employing (\ref{Uxip}), (\ref{Uxixip}) and (\ref{Phiphi}), we have, by Lemma 4.1,
\begin{equation} \label{4.1}
  \begin{aligned}
        \int_0^t\int_0^{\infty}\left[J\phi^2+m^2V_{\xi}^2\frac{\psi_{\xi}^2}{v}\right]\mathrm{d}\xi \mathrm{d}\tau%\\[2mm]
        \lesssim ~& \int_0^t\int_0^{\infty}\left[\left((m^{2\gamma+3}+M)V_{\xi}^2
              +mU_{\xi\xi}^2\right)\phi^2
              + m^2\epsilon^2\frac{\psi_{\xi}^2}{v}\right]\mathrm{d}\xi \mathrm{d}\tau\\[2mm]
        \lesssim ~&  M^{\gamma+1}(m^{2\gamma+3}+M)\epsilon^{\frac{2}{3}}
        \left[{\left\|(\sqrt{\Phi_0},\psi_0)\right\|}^2+{\epsilon}^{\frac{2}{9}}M^{\frac{1}{2}}\right].
  \end{aligned}
\end{equation}
For the last term on the right-hand side of (\ref{3rdEstimateRW}), we find from (\ref{3.1.2}) that
\begin{equation} \label{4.2}
  \begin{aligned}
    \int_0^t\left\|\frac{\phi_{\xi}}{v}(\tau)\right\|^4
    \left\|\frac{\psi_{\xi}}{\sqrt{v}}(\tau)\right\|^2\mathrm{d}\tau
    \lesssim \left[{\left\|(\sqrt{\Phi_0},\psi_0)\right\|}^2+{\epsilon}^{\frac{2}{9}}M^{\frac{1}{2}}\right]
    \left[\left(u_-^{-1}\int_0^t \psi_{\xi}^2(\tau,0)\mathrm{d}\tau\right)^2\right.\\[2mm]
    \left.+\epsilon^4m^4M^{2\gamma+2}\left({\left\|(\sqrt{\Phi_0},\psi_0)\right\|}^2
    +{\epsilon}^{\frac{2}{9}}M^{\frac{1}{2}}\right)^2
    +B^2\right],
  \end{aligned}
\end{equation}
according to Lemma 4.1 and Lemma 4.2.
Substitute (\ref{2ndEstimateRW}), (\ref{4.1}) and (\ref{4.2}) into (\ref{3rdEstimateRW}),
and then apply Cauchy's inequality to discover
\begin{equation}
  M(t)\lesssim A_1+A_2 M(t)^2,
\end{equation}
where
\begin{equation}
  \label{A1RW}\begin{aligned}
  A_1=~&
  \|\psi_{0\xi}\|^2+
  (m^{\gamma-1}+M^{\gamma-1}\epsilon^2)B+
  \frac{(m^{\gamma-1}+M^{\gamma-1}\epsilon^2)^2}{\|(\sqrt{\Phi_0},\psi_0)\|^2+{\epsilon}^{\frac{2}{9}}
  M^{\frac{1}{2}}}\\[2mm]
  &+M\epsilon+\left[{\left\|(\sqrt{\Phi_0},\psi_0)\right\|}^2+{\epsilon}^{\frac{2}{9}}M^{\frac{1}{2}}\right]
  \Bigg[M^{\gamma+1}(m^{2\gamma+3}+M)\epsilon^{\frac{2}{3}}\\[2mm]
  &+\epsilon^4m^4M^{2\gamma+2}\left({\left\|(\sqrt{\Phi_0},\psi_0)\right\|}^2
    +{\epsilon}^{\frac{2}{9}}M^{\frac{1}{2}}\right)^2+B^2\Bigg]
  \end{aligned}
\end{equation}
and
\begin{equation}
  \label{A2RW}
  A_2=u_-^{-4}\left[\left\|(\sqrt{\Phi_0},\psi_0)\right\|^2+{\epsilon}^{\frac{2}{9}}M^{\frac{1}{2}}\right].
\end{equation}
Then employing Lemma \ref{Strauss}, we conclude the following lemma.
\begin{Lemma}   There is a  constant $c_3$ independent of $T$, $m$, $M$, $\delta$ and $a$, such that
  if $\epsilon$ is suitably small and
  \begin{equation}
    \label{conditionRW}
    A_1A_2<c_4,
  \end{equation}
  where $A_1$ and $A_2$ are defined by (\ref{A1RW}) and (\ref{A2RW}), respectively.
  Then it holds that for each $0\leq t\leq T$,
  \begin{equation}
    \label{3rdEstimate1RW}
    \|\psi_{\xi}(t)\|^2+u_-\int_0^t\psi_{\xi}^2(\tau,0)\mathrm{d}\tau
    +\int_0^t\left\|\frac{\psi_{\xi\xi}}{\sqrt{v}}(\tau)\right\|^2\mathrm{d}\tau
    \lesssim A_1
  \end{equation}
  and
  \begin{equation} \label{2ndEstimate1RW}
\begin{aligned}
    \left\|\frac{\tilde{v}_{\xi}}{\tilde{v}}(t)\right\|^2
    +\int_0^t\int_0^{\infty}\frac{\tilde{v}_{\xi}^2}{v^{\gamma}\tilde{v}^2}\mathrm{d}\xi \mathrm{d}\tau
    \lesssim B+u_-^{-2}A_1,
\end{aligned}
\end{equation}
where $B$ is defined by (\ref{B}).
\end{Lemma}

\noindent
\textbf{Proof of Theorem 3.}
Since $({\phi}_0,{\psi}_0)\in H_0^1(\mathbb{R}_{+})$,
we find $t_0>0$ from Lemma \ref{localsRW} such that
(\ref{RWperturbation}) has a unique solution $(\phi,\psi)\in X_{1/m_0,M_0}(0,t_0)$
with $m_0,M_0\leq C\epsilon^{-l}$.
Then we deduce that if (\ref{indexR})$_1$ holds, then
$\|\sqrt{\Phi_0},\psi_0\|^2+{\epsilon}^{{2}/{9}}M_0^{{1}/{2}}$ and $B$ are bounded by
$
    C_1(1+\epsilon^{-2\alpha-(\gamma+1)l}).
$
Then we compute that if (\ref{indexR})$_1$ holds, then
$A_1$ and $A_2$ are, respectively, bounded by
$C_2(1+\epsilon^{-6\alpha-3(\gamma+1)l})$ and $C_3(1+\epsilon^{4l_0-2\alpha-(\gamma+1)l}).$
Hence if $(\ref{indexR})_2$ holds, there exists $\epsilon_1>0$ such that
(\ref{conditionRW}) holds for each $\epsilon \leq \epsilon_1$.
Next we compute from
$(\ref{indexR})_{2}$ that the right-hand side of (\ref{2ndEstimate1RW})
is bounded by $C_4(1+\epsilon^{-2\alpha-(\gamma+1)l})$.
We conclude from Lemma 4.1 and Lemma 4.3 that
\begin{equation}
  \begin{aligned}
    \left|\Psi(\tilde{v}(t_0,\xi))\right|
    \leq \bigg\| \sqrt{\tilde{\Phi}}(t_0) \bigg\| \bigg\| \frac{\tilde{v}_{\xi}}{\tilde{v}}(t_0) \bigg\|
    \leq C\epsilon^{\theta},
  \end{aligned}
\end{equation}
where $\theta=-2\alpha-(\gamma+1)l.$
Hence
\begin{equation} \label{boundofvR}
  C^{-1}\epsilon^{{2\theta}/(1-\gamma)}\leq v (t_0)\leq C\epsilon^{2\theta}.
\end{equation}

Since $(\ref{indexR})_3$ implies $2\theta/(\gamma-1)\leq -l$ and $ 2\theta \leq -l$,
we apply Lemma \ref{localsRW} again and recall (\ref{boundofvR}) to find $t_1>0$
such that (\ref{BLperturbation}) has a unique solution $(\phi,\psi)\in X_{1/m_1,M_1}(0,t_0+t_1)$,
where $m_1 \leq  C \epsilon^{2\theta/(\gamma-1)}$ and $ M_1 \leq C \epsilon ^{2\theta}$.
By elementary calculations, we conclude from $(\ref{indexR})_4$ that
$\|\sqrt{\Phi_0},\psi_0\|^2+{\epsilon}^{{2}/{9}}M_0^{{1}/{2}}$ and $B$ are bounded by
$
    C_1(1+\epsilon^{-2\alpha-(\gamma+1)l}).
$
Then we compute that if (\ref{indexR})$_4$ holds, then
$A_1$ and $A_2$ are, respectively, bounded by
$C_5(1+\epsilon^{-6\alpha-3(\gamma+1)l})$ and $C_6(1+\epsilon^{4l_0-2\alpha-(\gamma+1)l}).$
Hence if $(\ref{indexR})_2$ holds, there exists $\epsilon_0>0$ such that
(\ref{conditionRW}) holds for each $\epsilon \leq \epsilon_0$.
Thus, we find from $(\ref{indexR})_2$ that the right-hand side of (\ref{2ndEstimate1RW})
is bounded by $C_4(1+\epsilon^{-2\alpha-(\gamma+1)l})$.
Then, combining Lemma \ref{localsRW} and the continuation process,
we can prove (\ref{BL}) has the global solution
in time $(\phi,\psi)\in X_{1/m_1,M_1}(0,\infty)$
satisfying
(\ref{estimate}) with the constant $C$ depending only on $\epsilon$.
Thus, the asymptotic behavior of the solution (\ref{largetime})
is concluded by employing Sobolev's inequality.$\quad\quad \Box$

\bigbreak

\begin{center}
{\bf Acknowledgement}
\end{center}
This work was supported by ``the Fundamental Research Funds for the Central
Universities" and four grants
from the National Natural Science Foundation of China under
contracts 10925103, 11271160, 11261160485, and 11301405, respectively.
This work was completed when Tao Wang was visiting the Mathematical
Institute at the University of Oxford under the support of the China Scholarship
Council 201206270022. He would like to thank Professor Gui-Qiang Chen and his
group for their kind hospitality.


\small

\begin{thebibliography}{99}

% \bibitem{Fan-Yin-Zhao} L. L. Fan, H. Yin and H. J. Zhao, Decay rates toward stationary waves of solutions for damped wave equations. J. Partial Differential Equations 21 (2008), no. 2, 141-172.
% \bibitem{Hashimoto} I. Hashimoto, Large-time behavior of solutions of scalar viscous conservation law with non-convex flux. PhD thesis, Graduate School of Information Science and Technology, Osaka University, January 2009.
 \bibitem{H-M-S} F. M. Huang, A. Matsumura and X. D. Shi, Viscous shock wave and boundary layer to an inflow problem for compressible viscous gas. {\it Comm. Math. Phys.} {\bf 239} (2003), no. 1-2, 261-285.
 \bibitem{H-Q} F. M. Huang and X. H. Qin, Stability of boundary layer and rarefaction wave to an outflow problem for compressible Navier-Stokes equations under large perturbation. {\it J. Differential Equations} {\bf 246} (2009), no. 10, 4077-4096.
% \bibitem{Kagei-Kawashima} Y. Kagei and S. Kawashima, Stability of planar stationary solutions to the compressible Navier-Stokes equation on the half space. Comm. Math. Phys. 266 (2006), no. 2, 401-430.
 \bibitem{Kanel'} Ya. Kanel', A model system of equations for the one-dimensional motion of a gas. (Russian) {\it Differencial'nye Uravnenija} {\bf 4} (1968), 721-734; English transl. in {\it Diff. Eqns.} {\bf 4} (1968), 374-380.
% \bibitem{Kawashima-Kurata} S. Kawashima and K. Kurata, Hardy type inequality and application to the stability of degenerate stationary waves. J. Funct. Anal. 257 (2009), no. 1, 1-19.
 \bibitem{Kawashima-Nishibata-Zhu} S. Kawashima, S. Nishibata and P. C. Zhu, Asymptotic stability of the stationary solution to the compressible Navier-Stokes equations in the half space. {\it Comm. Math. Phys.} {\bf 240} (2003), no. 3, 483-500.
 \bibitem{Kawahima-Zhu} S. Kawashima and P. C. Zhu, Asymptotic stability of nonlinear wave for the compressible Navier-Stokes equations in the half space. {\it J. Differential Equations} {\bf 244} (2008), no. 12, 3151-3179.
% \bibitem{Liu-Matsumura-Nishihara} T.-P. Liu, A. Matsumura and K. Nishihara, Behaviors of solutions for the Burgers equation with boundary corresponding to rarefaction waves. SIAM J. Math. Anal. 29 (1998), no. 2, 293-308.
% \bibitem{Liu-Nishihara} T.-P. Liu and K. Nishihara, Asymptotic behavior for scalar viscous conservation laws with boundary effect. J. Differential Equations 133 (1997), no. 2, 296-320.
 \bibitem{M-MAA-2001} A. Matsumura, Inflow and outflow problems in the half space for a one-dimensional isentropic model system of compressible viscous gas. IMS Conference on Differential Equations from Mechanics (Hong Kong, 1999). {\it Methods Appl. Anal.} {\bf 8} (2001), no. 4, 645-666.
 \bibitem{M-Mei} A. Matsumura and M. Mei, Convergence to travelling fronts of solutions of the $p$-system with viscosity in the presence of a boundary. {\it Arch. Ration. Mech. Anal.} {\bf 146} (1999), no. 1, 1-22.
% \bibitem{M-N1985} A. Matsumura and K. Nishihara, On the stability of travelling wave solutions of a one-dimensional model system for compressible viscous gas. Japan J. Appl. Math. 2 (1985), no. 1, 17-25.
% \bibitem{M-N1986} A. Matsumura and K. Nishihara, Asymptotics toward the rarefaction waves of the solutions of a one-dimensional model system for compressible viscous gas. Japan J. Appl. Math. 3 (1986), no. 1, 1-13.
 \bibitem{M-N1992} A. Matsumura and K. Nishihara, Global stability of the rarefaction wave of a one-dimensional model system for compressible viscous gas. {\it Comm. Math. Phys.} {\bf 144} (1992), no. 2, 325-335.
% \bibitem{M-N1994} A. Matsumura and K. Nishihara, Asymptotic stability of traveling waves for scalar viscous conservation laws with non-convex nonlinearity. Comm. Math. Phys. 165 (1994), no. 1, 83-96.
 \bibitem{M-N2000} A. Matsumura and K. Nishihara, Global asymptotics toward the rarefaction wave for solutions of viscous $p$-system with boundary effect. {\it Quart. Appl. Math.} {\bf 58} (2000), no. 1, 69-83.
 \bibitem{M-N2001} A. Matsumura and K. Nishihara, Large-time behaviors of solutions to an inflow problem in the half space for a one-dimensional system of compressible viscous gas. {\it Comm. Math. Phys.} {\bf 222} (2001), no. 3, 449-474.
 \bibitem{Nakamura-Nishibata-Yuge} T. Nakamura, S. Nishibata and T. Yuge, Convergence rate of solutions toward stationary solutions to the compressible Navier-Stokes equation in a half line. {\it J. Differential Equations} {\bf 241} (2007), no. 1, 94-111.
 \bibitem{Nikkuni-Kawashima} Y. Nikkuni and S. Kawashima, Stability of stationary solutions to the half-space problem for the discrete Boltzmann equation with multiple collisions. {\it Kyushu J. Math.} {\bf 54} (2000), no. 2, 233-255.
% \bibitem{Nishihara} K. Nishihara, Boundary Effect on a Stationary Viscous Shock Wave for Scalar Viscous Conservation Laws. J. Math. Anal. Appl. 255 (2001), 535-550.
\bibitem{Qin} X. H. Qin, Ph. D. Thesis, Institute of Applied Mathematics, Academy of Mathematics and System Sciences, the Chinese Academy of Sciences (in Chinese).
 \bibitem{Qin-Wang} X. H. Qin and Y. Wang, Stability of wave patterns to the inflow problem of full compressible Navier-Stokes equations. {\it SIAM J. Math. Anal.} {\bf 41} (2009), no. 5, 2057-2087.
 \bibitem{Shi} X. D. Shi, On the stability of rarefaction wave solutions for viscous $p$-system with boundary effect. {\it Acta Math. Appl. Sin. Engl. Ser.} {\bf 19} (2003), no. 2, 341-352.
 \bibitem{Strauss} W. A. Strauss, Decay and asymptotics for $u_{tt}-\triangle u=F(u)$. {\it J. Functional Analysis} {\bf 2} (1968), 409-457.
% \bibitem{Ueda} Y. Ueda, Yoshihiro Asymptotic stability of stationary waves for damped wave equations with a nonlinear convection term. Adv. Math. Sci. Appl. 18 (2008), no. 1, 329-343.
% \bibitem{Ueda-Nakamura-Kawashima1} Y. Ueda, T. Nakamura and S. Kawashima, Stability of planar stationary waves for damped wave equations with nonlinear convection in multi-dimensional half space. Kinet. Relat. Models 1 (2008), no. 1, 49-64.
% \bibitem{Ueda-Nakamura-Kawashima2} Y. Ueda, T. Nakamura and S. Kawashima, Stability of degenerate stationary waves for viscous gases. Arch. Ration. Mech. Anal. 198 (2010), no. 3, 735-762.
% \bibitem{Yin-Zhao} H. Yin and H. J. Zhao, Nonlinear stability of boundary layers for generalized Benjamin-Bona-Mahony-Burgers equation in the half space. Kinet. Relat. Models 2 (2009), no. 3, 521-550.
% \bibitem{Yin-Zhao-Kim} H. Yin, H. J. Zhao and J. S. Kim, Convergence rates of solutions toward boundary layers for generalized Benjamin-Bona-Mahony-Burgers equations in the half-space. J. Differential Equations 245 (2008), no. 11, 3144-3216.
% \bibitem{Zhu-report} P. C. Zhu, Nonlinear Waves for the Compressible Navier-Stokes Equations in the Half Space. The report for JSPS postdoctoral research at Kyushu University, Fukuoka, Japan, August 2001.

\end{thebibliography}
\end{document}